\documentclass[12pt,reqno,a4paper]{article}
\usepackage{xcolor}

\usepackage{authblk}
\usepackage{euscript}
\usepackage{amsmath,amssymb,amsfonts,amsthm}
\usepackage{mathtools}
\usepackage{fullpage}
\usepackage{dsfont}
\usepackage{sectsty}
\usepackage[normalem]{ulem}
\usepackage[colorlinks=true, linkcolor=blue]{hyperref}
\usepackage{comment}

\newtheoremstyle{special}%
{}%
{}%
{}%
{}%
{\scshape}%
{.}%
{.5em}%
{}

\newtheorem{theorem}{Theorem}
\newtheorem{proposition}[theorem]{Proposition}
\newtheorem{lemma}[theorem]{Lemma}
\newtheorem{cor}[theorem]{Corollary}

\theoremstyle{special}
\newtheorem{remark}[theorem]{Remark}

\renewcommand{\epsilon}{\varepsilon}
\renewcommand{\L}{\mathcal{L}}

\def\Id{\text{\rm Id}}

\def\N{\mathbb{N}}
\def\Z{\mathbb{Z}}
\def\R{\mathbb{R}}

\def\C{\mathcal C}

\DeclareMathOperator{\esssup}{esssup}
\DeclareMathOperator{\essinf}{essinf}

\title{Linear response for random and sequential intermittent maps}

\author{Davor Dragi\v cevi\' c
\thanks{Faculty of Mathematics, University of Rijeka, Radmile Matej\v ci\' c 2, 51000 Rijeka, Croatia.\\ Email: ddragicevic@math.uniri.hr}
\and
Cecilia Gonz\'alez-Tokman \thanks{School of Mathematics and Physics, The University of Queensland, Brisbane, QLD 4072, Australia.\\ Email: cecilia.gt@uq.edu.au}
\and
Julien Sedro \thanks{Email: sedro.math@proton.me}
}

\begin{document}
	\maketitle

 \begin{abstract}
This work establishes a quenched (trajectory-wise) linear response formula for random intermittent dynamical systems, consisting of Liverani-Saussol-Vaienti maps with varying parameters. This result complements recent annealed (averaged) results in the i.i.d setting. 
As an intermediate step, we show existence, uniqueness and statistical stability of the random absolutely continuous invariant probability measure (a.c.i.m.) for such non-uniformly expanding systems. 
Furthermore, we investigate sequential
intermittent dynamical systems of this type and establish a linear response formula. 
Our arguments rely on the cone technique introduced by Baladi and Todd and further developed by Lepp{\"a}nen.
 We also demonstrate that sequential systems exhibit a subtle distinction from both random and autonomous settings: they may possess infinitely many sequential absolutely continuous equivariant densities. However, only one of these corresponds to an SRB state in the sense of Ruelle.

    \vspace{3mm}

\vspace{2mm}

 \end{abstract}
	\section{Introduction}

Intermittent maps, such as the ones introduced by Liverani, Saussol and Vaienti in \cite{LSV} (called LSV maps), are widely studied non-uniformly expanding systems exhibiting a combination of chaotic and regular behavior. Trajectories of these systems behave chaotically on a large part of the state space, but when they visit small neighbourhoods of the neutral fixed point at the origin, it takes a long time before they return to the chaotic regime. The ergodic and statistical properties of intermittent systems have been widely investigated in the last few decades. We refer to~\cite{alves2020nonuniformly} for a detailed survey. 
 Another type of intermittency, called critical intermittency, has been recently investigated in 
 \cite{HomburgEtAl22, KalleZeegers23}.

Linear response takes place when a 
(unique) physical measure of a system changes differentiably with respect to perturbations of the dynamics. This is a strong form of stability under parameter changes, whose rigorous mathematical study has been initiated 
by Ruelle~\cite{Ruelle-diffSRB}. Since then, the existence (and the lack of) linear response has been investigated  for wide classes of deterministic dynamical systems including: smooth dynamical systems~\cite{baladi2018dynamical,sedro2018regularity}, piecewise expanding maps on the interval~\cite{baladi2007susceptibility, baladi2008linear}, uniformly hyperbolic diffeomorphisms and flows~\cite{butterley2007smooth, butterley2013robustly, gouezel2006banach, Ruelle-diffSRB} as well as partially hyperbolic systems~\cite{dolgopyat2004differentiability}. 

The linear response for LSV maps has been obtained in three independent works that appeared around the same time. Namely, it was established by Baladi and Todd~\cite{BT} by relying on cone techniques, and by  Bahsoun and Saussol~\cite{bahsoun2016linear} as well as  Korepanov~\cite{korepanov2016linear} using inducing type arguments. More recently, Lepp{\"a}nen~\cite{L} has refined the Baladi-Todd approach and established linear response for a family of maps with a neutral fixed point and a non-flat critical point, which includes LSV maps.

The main objective of the present paper is to establish  linear response for LSV maps in random and nonautonomous (sequential) contexts. More precisely, we deal with two-sided sequences of maps $(T_n)_{n\in \Z}$, where each $T_n$ is an LSV map.
The dynamics is generated by compositions of the form
\[
T_n^m:=T_{n+m-1}\circ \ldots \circ T_{n+1}\circ T_n \quad \text{for $n\in \Z$ and $m\in \N$.}
\]
 In the case of random dynamics, map sequences are obtained from trajectories of an invertible, ergodic, probability preserving map $\sigma \colon (\Omega, \mathbb P) \to (\Omega, \mathbb P)$. That is $T_n=T_{\sigma^n \omega}$, $\omega \in \Omega, \ n\in \mathbb Z$.

A (nonautonomous) absolutely continuous invariant measure (a.c.i.m) for such system will be a sequence $(\mu_n)_{n\in \mathbb Z}$ of probability measures on the unit interval  which are absolutely continuous with respect to the Lebesgue measure $m$ and satisfy that 
\[
(T_n)_*\mu_n=\mu_{n+1} \quad n\in \mathbb Z,
\]
where $(T_n)_*\mu_n$ denotes the push-forward of $\mu_n$ with respect to $T_n$. This concept is related to the concept of SRB states in the nonautonomous setting, discussed by Ruelle~\cite[Section 4]{Ruelle-diffSRB}, although the latter requires an extra condition. 
Namely, SRB states are the limit of Lebesgue measure, $m$, pushed forward from the infinite past, $\mu_k=\lim_{n\to \infty} (T_{k-n})_* (T_{k-n+1})_*\dots (T_{k-1})_* m$.
 We show that this extra condition is essential to ensure uniqueness, as otherwise sequential systems often have multiple sequential a.c.i.m.s.  Another similar concept is presented by Young in~\cite[Section 4]{young2017generalizations}, although that work deals with one-sided sequences of maps. 
 
Similarly, a random a.c.i.m $(\mu_\omega)_{\omega \in \Omega}$ satisfies $(T_\omega)_*\mu_\omega=\mu_{\sigma\omega}$ for  $\mathbb P$-$\text{a.e. } \omega \in \Omega$, and 
an extra measurability condition, namely that $(\mu_\omega)_{\omega \in \Omega}$ is a random measure. That is, it provides the disintegration of a measure $\mu$ on $\Omega \times [0, 1]$, with marginal $\mathbb P$.

Our goal is to study the dependence of random a.c.i.ms $(\mu_\omega)_{\omega \in \Omega}$, respectively sequencial a.c.i.ms $(\mu_n)_n$, under small perturbations of the family $(T_\omega)_{\omega \in \Omega}$, respectively $(T_n)_{n\in \Z}$. More precisely, in the random setting, suppose that $I$ is a small interval around $0$ in $\mathbb R$ and that we are given a parametrized family of LSV maps $(T_{\omega, \epsilon})_{\omega \in \Omega}$, $\epsilon \in I$ where we view $T_{\omega,  \epsilon}$ as a ``small'' perturbation of $T_\omega:=T_{\omega,  0}$. Since both of these maps belong to the LSV family,  this means that the parameter associated to $T_{\omega,  \epsilon}$ is close to the parameter associated to $T_\omega$.  By $(\mu_{\omega,  \epsilon})_{\omega \in \Omega}$ we will denote the (unique) random a.c.i.m corresponding to the sequence $(T_{\omega, \epsilon})_{\omega \in \Omega}$, provided by Proposition~\ref{p1}.
We are then interested in formulating sufficient conditions under which, for $\mathbb P$-$  \text{a.e. } \omega \in \Omega$ and suitable observable $\phi$, the map $\epsilon \to \int \phi\, d\mu_{\omega,  \epsilon}$ is continuous (quenched statistical stability; see Theorem~\ref{SS}) or differentiable at $0$ 
(quenched linear response; see Theorems~\ref{thm:LR} and \ref{thm:LqLR}). The sequential setting is   treated in Theorem~\ref{thm:seqLR}.
We note that there is  a series of recent results  devoted to the quenched linear response for various classes of random dynamical systems~\cite{crimmins2024spectral,  dragivcevic2023quenched, dragicevic2024effective, dragivcevic2023statistical, sedro2021regularity}, although in these works the case of random intermittent dynamics was not considered. Hence, our main result complements the recent annealed linear response result for random i.i.d compositions of LSV maps obtained in~\cite{bahsoun2020linear}.

On the other hand, to the best of our knowledge, Ruelle~\cite[Section 4]{Ruelle-diffSRB} was the first to discuss linear response in nonautonomous (sequential) setting. He
briefly outlined a strategy to obtain linear response for a class of dynamics formed by two-sided sequences of maps $(T_n)_{n\in \mathbb Z}$, where each $T_n$ is a sufficiently small perturbation of a fixed uniformly hyperbolic map, and established a linear response formula, valid when a fixed map is repeatedly applied, except for a finite number of steps, where (small) perturbations take place. A possibility for establishing linear response for sequential expanding dynamics was also indicated by Sedro and Rugh~\cite{sedro2021regularity}.
The present work is the first one which considers the linear response problem in this 
nonautonomous (sequential) nonuniformly hyperbolic  setting. Previous works on this type of dynamics have been devoted either to decay of correlations or to limit theorems (see~\cite{AHNTV, KL, nicol2018central, nicol2021large} and references therein). 

Our arguments rely on the work of  Lepp{\"a}nen~\cite{L}, which in turn is inspired by the earlier work of Baladi and Todd~\cite{BT}. 
However, nontrivial technical changes are necessary to account for the non-autonomous nature of the problem under consideration.

The paper is organized as follows.  Random absolutely continuous invariant measures are identified in Section~\ref{S:SeqDens}, and their quenched statistical stability is established in Section~\ref{S:SeqStatStab}. Section~\ref{S:LinRep} presents results on linear response for bounded observables, including the random setting in Theorem~\ref{thm:LR} and Corollary~\ref{cor:ALR}, and the sequential setting in Theorem~\ref{thm:seqLR}. A general result on non-uniqueness of sequential a.c.i.m.s is presented in Section~\ref{ssec:seq}, Proposition~\ref{L:non-uniq}. 
Section~\ref{sec:lqObs} establishes linear response for random intermittent maps for a class of $L^q$ observables.
 
\paragraph{\textbf{Convention.}} Throughout this work, $C$ and $D$ will denote positive constants whose precise values may change from one appearance to the next. Dependence on parameters such as $\alpha$ will be indicated with subscripts, e.g. $C=C_\alpha$, as needed.

	\section{Existence and uniqueness of random equivariant density}\label{S:SeqDens}
 \subsection{Preliminaries}
	For each parameter $\alpha \in [0,1)$, we consider the associated LSV-map $T_\alpha \colon [0,1] \to [0,1]$ given by 
	\[
	T_\alpha (x)=\begin{cases}
		x(1+2^\alpha x^\alpha) & 0 \le x<\frac 1 2; \\
		2x-1 & \frac 1 2 \le x \le 1.
	\end{cases}
	\]
	The transfer operator associated to $T_\alpha$ will be denoted by 
	$\mathcal L_\alpha \colon L^1(m) \to L^1(m)$, where $m$ denotes the Lebesgue measure on $[0,1]$. We recall that 
	\[
	(\mathcal L_\alpha \phi)(x)=\sum_{T_\alpha(y)=x}\frac{\phi(y)}{T_\alpha'(y)}, \quad x\in [0,1], \ \phi\in L^1(m).
	\]
We have that 	\begin{equation}\label{dual}
		\int_0^1 (\mathcal L_\alpha \phi)\psi\, dm=\int_0^1 \phi (\psi\circ T_\alpha)\, dm,
	\end{equation}
	for $\phi\in L^1(m)$ and $\psi\in L^\infty(m)$.  In particular, it follows from~\eqref{dual} applied to $\psi\equiv 1$ that 
	\begin{equation}\label{intpres}
		\int_0^1 \mathcal L_\alpha \phi\, dm=\int_0^1 \phi\, dm, \quad \phi\in L^1(m).
	\end{equation}
 
	Denoting, for $0 \le x\le \frac 1 2$, by $f_\alpha(x)=x(1+2^\alpha x^\alpha)\in[0,1]$ the first branch of $T_\alpha$, we introduce its inverse $g_\alpha=f_{\alpha}^{-1}:[0,1]\to[0,1/2]$. It is easy to see that, for $\alpha\in[0,1]$, $1\le f'_\alpha(x)\le 3$, so that
	\begin{equation}\label{eq:trivialboundong}
	\frac 13\le g'_\alpha(x)\le 1\quad\text{and}\quad \frac x3\le g_{\alpha}(x)\le x.
	\end{equation}
	Let $N_\alpha$ be the  transfer operator associated to $f_\alpha$. Then,
	\begin{equation}
		N_\alpha\phi:= g_\alpha'\phi\circ g_\alpha.
	\end{equation}
	Following \cite{BT}, we will also denote $v_\alpha(x):=\partial_\alpha T_\alpha(x)=2^\alpha x^{1+\alpha}\log(2x)$ and 
	\begin{equation}\label{def:mapperturb}
	X_\alpha:=v_\alpha\circ g_\alpha=2^\alpha g_\alpha^{1+\alpha}\log(2g_\alpha).
	\end{equation}
	
 \subsection{Cones}
	Let $X$ denote the identity map on $[0,1]$ and fix $\alpha \in (0,1)$.
	For $a>1$ we introduce the convex and closed cone $\mathcal C_*(a)$ that consist of all $\phi\in C^0(0,1] \cap L^1(m)$ with the following properties:
	\begin{itemize}
		\item $\phi\ge 0$ and $\phi$ decreasing;
		\item $X^{\alpha+1}\phi$ increasing;
		\item $\int_0^x \phi(t)\, dt \le ax^{1-\alpha}m(\phi)$ for $0<x\le 1$, where $m(\phi)=\int_0^1 \phi(x)\, dm(x)$.
	\end{itemize}

\begin{remark}\label{n06}
\begin{itemize}
\item Let $\phi \in \mathcal C_*(a)$. Since $\phi$ is decreasing,
\[
x\phi(x)\le \int_0^x \phi(t)\, dt \le ax^{1-\alpha}m(\phi),
\]
and consequently
\[
\phi(x)\le ax^{-\alpha}m(\phi), \quad 0<x\le 1.
\]
\item Take $\phi \in \C_\ast(a)\cap C^1(0, 1]$. Since $X^{\alpha+1}\phi$ is increasing and $\phi$ is decreasing, we have that 
\[
(\alpha+1)x^\alpha\phi(x) +x^{\alpha+1}\phi'(x) \ge 0 \quad \text{and} \quad \phi'(x)\le 0,  \quad x\in (0, 1].
\]
Consequently, 
\[
|\phi'(x)| \le \frac{\alpha+1}{x}\phi(x), \quad x\in (0, 1].
\]
\end{itemize}
\end{remark}

\begin{lemma}\label{lem2}
Let $\alpha \in (0, 1)$. Then, there exists $a=a(\alpha)>1$ such that  
\[
\L_\beta (\mathcal C_*(a))\subset \mathcal C_*(a), \quad 0<\beta \le \alpha.
\]
\end{lemma}

\begin{proof}
It follows from~\cite[Lemma 2.1]{L} applied to $\gamma=(\alpha, 1)$.
\end{proof}

The following result is similar to~\cite[Lemma 3.3]{L}. 
\begin{lemma}\label{insidecone}
Let $F\in C^2(0, 1]$ be such that $m(F)=0$ and with the property that there exist $C_i>0$ for $i\in \{1, 2\}$ and $0<\delta \le \alpha$ so that
\[
|F(x)| \le C_1ax^{-\delta} \quad \text{and} \quad |F'(x)| \le C_2ax^{-\delta-1},
\]
for $x\in (0, 1]$. Then, $F+\lambda x^{-\alpha}$ belongs to $\mathcal C_*(a)$ if 
\[
\lambda \ge \max \left \{ C_1 a, \frac{C_2a}{\alpha}, (2C_1+C_2)a, \frac{C_1a}{a-1} \right \}.
\]
\end{lemma}

\begin{proof}
Let 
\[
\psi(x):=F(x)+\lambda x^{-\alpha}, \quad x\in (0, 1].
\]
Observe that 
\[
\psi(x)\ge \lambda x^{-\alpha}-|F(x)| \ge \lambda x^{-\alpha}-C_1ax^{-\delta}\ge (\lambda-C_1a)x^{-\alpha},
\]
for $x\in (0, 1]$. Hence, $\psi \ge 0$ provided that $\lambda \ge C_1a$.

Moreover, we have that $\psi'(x)=F'(x)-\lambda \alpha x^{-\alpha -1}$ and consequently
\[
\psi'(x) \le |F'(x)|-\lambda \alpha x^{-\alpha -1} \le C_2ax^{-\delta-1}-\lambda \alpha x^{-\alpha -1} \le (C_2a-\lambda \alpha)x^{-\alpha -1},
\]
for $x\in (0, 1]$. Thus, $\psi$ is decreasing if $\lambda \ge \frac{C_2a}{\alpha}$.

Moreover, 
\[
(X^{\alpha+1}\psi)'(x)=(\alpha+1)x^\alpha F(x)+x^{\alpha+1}F'(x)+\lambda,
\]
and thus
\[
(X^{\alpha+1}\psi)'(x) \ge \lambda-(\alpha+1)C_1ax^{\alpha-\delta}-C_2ax^{\alpha-\delta}\ge \lambda -2C_1a-C_2a,
\]
for $x\in (0, 1]$. Therefore, $X^{\alpha+1}\psi$ is increasing provided that $\lambda \ge (2C_1+C_2)a$.

Since $m(F)=0$ we have that $m(\psi)=\frac{\lambda}{1-\alpha}$. In addition,
\[
\begin{split}
\int_0^x \psi (t)\, dt &=\int_0^xF(t)\, dt+\lambda \int_0^x t^{-\alpha}\, dt \\
&\le C_1a\int_0^x t^{-\delta}\, dt+\frac{\lambda}{1-\alpha}x^{1-\alpha} \\
&\le \frac{C_1a}{1-\delta}x^{1-\delta}+\frac{\lambda}{1-\alpha}x^{1-\alpha} \\
&\le \frac{C_1 a+\lambda}{1-\alpha}x^{1-\alpha}\\
&=\frac{C_1a+\lambda}{\lambda}x^{1-\alpha} m(\psi),
\end{split}
\]
for $x\in (0, 1]$. We conclude that $\int_0^x\psi(t)\, dt \le ax^{1-\alpha}m(\psi)$ for $x\in (0, 1]$ if 
\[
\lambda \ge \frac{C_1a}{a-1}.
\]
\end{proof}

\begin{remark}\label{181}
Observe that $\lambda x^{-\alpha}$ belongs to $\mathcal C_*(a)$ provided that  $\lambda \ge 0$.
\end{remark}

For $b_i>0$, $i=1, 2$ we introduce the  convex  cone $\C_2(b_1, b_2)$ which consists of all $\phi \in C^2(0,1]$ such that 
	\[
\phi(x)\ge 0, \quad  |\phi'(x)| \le \frac{b_1}{x}\phi(x)   \quad \text{and} \quad  |\phi''(x)|\le \frac{b_2}{x^2}\phi(x), \quad \forall x\in (0,1].
	\]
Furthermore, for an additional parameter $b_3>0$, we consider the convex cone
\[
 \C_3(b_1, b_2, b_3):=\left \{\phi \in C^3(0,1]: \ \phi \in \C_2(b_1, b_2) \quad  \text{and} \quad  
|\phi^{(3)}(x)| \le \frac{b_3}{x^3}\phi(x), \ \forall x\in (0,1] \right \}.
 \]
We have the following result which is similar  to~\cite[Proposition 2.4]{BT}.
	\begin{proposition}\label{prop:conecontraction}
 Take $\alpha \in (0,1)$. Then, there exist $b_k$, $k\in\{1, 2,3\}$ such that for $0< \beta \le \alpha$,
			\begin{equation}\label{eq:conecontractionI}
			\L_\beta(\C_j)\subset \C_j~ \text{and}~N_\beta(\C_j)\subset \C_j,
			\end{equation}
   for $j\in \{2, 3\}$, where $\C_2=\C_2(b_1, b_2)$ and $\C_3=\C_3(b_1, b_2, b_3)$.

	\end{proposition} 

\begin{remark}
In comparison to~\cite{BT}, in Proposition~\ref{prop:conecontraction} we show that the cones $\mathcal C_j$, $j\in \{2,3\}$ (defined for parameters that depend only on $\alpha$) are preserved for $\L_\beta$ and $N_\beta$ for each $0<\beta \le \alpha$. This is essentially achieved  by closely inspecting the arguments in~\cite{BT}.
\end{remark}

	\begin{proof}[Proof of Proposition~\ref{prop:conecontraction}] 
        Take $b_1\ge 1+\alpha$.
		As $(\L_\beta-N_\beta)\phi=\phi((x+1)/2)/2$, it is enough to prove the second inclusion in~\eqref{eq:conecontractionI}.
Indeed, assuming that the second inclusion in~\eqref{eq:conecontractionI} holds, we have 
\[
\begin{split}
|(\L_\beta \phi)'(x)| \le |(N_\beta \phi)'(x)|+\frac 1 4 |\phi'((x+1)/2)| &\le \frac{b_1}{x}N_\beta \phi(x)+\frac{b_1}{2(x+1)}\phi((x+1)/2) \\
&\le \frac{b_1}{x} \left (N_\beta \phi(x)+\frac 1 2 \phi((x+1)/2) \right )\\
&=\frac{b_1}{x} \L_\beta \phi(x),
\end{split}
\]
and thus $|(\L_\beta \phi)'(x)| \le \frac{b_1}{x}\L_\beta \phi(x)$ for $0<x\le 1$. Similarly, one can treat $| (\L_\beta \phi)^{(j)}(x)|$ for $j\in \{2, 3\}$.

We now turn to the second inclusion in~\eqref{eq:conecontractionI}. For $\phi\in\C_2$ and $x\in (0, 1]$, by writing $y=g_\beta(x)$, we have (see~\cite[p.869]{BT}) that 
\[
\begin{split}
|(N_\beta \phi)'(x)| \le \frac{b_1}{x}N_\beta \phi (x)\sup_{y\in [0, \frac 1 2]}\left [\frac{T_\beta(y)}{b_1}\cdot \left (\frac{T_\beta''(y)}{(T_\beta'(y))^2}+\frac{b_1}{yT_\beta'(y)}\right )\right].
\end{split}
\]
Moreover,
\begin{equation}\label{5:02}
\begin{split}
\frac{T_\beta(y)}{b_1}\cdot \left (\frac{T_\beta''(y)}{(T_\beta'(y))^2}+\frac{b_1}{yT_\beta'(y)}\right  ) &=\frac{T_\beta(y)}{yT_\beta'(y)}\left(1+\frac{yT''_\beta(y)}{b_1T'_\beta(y)}\right) \\
&=\frac{1+2^\beta y^\beta}{1+2^\beta(1+\beta)y^\beta}\left(1+\frac{2^\beta\beta(1+\beta)y^\beta}{b_1(1+2^\beta(1+\beta)y^\beta)}\right)\\
&=\left(1-\frac{2^\beta\beta y^\beta}{1+2^\beta(1+\beta)y^\beta}\right)\left(1+\frac{2^\beta\beta(1+\beta)y^\beta}{b_1(1+2^\beta(1+\beta)y^\beta)}\right)\\
&\le \left(1-\frac{2^\beta\beta y^\beta}{1+2^\beta(1+\beta)y^\beta}\right)\left(1+\frac{2^\beta\beta y^\beta}{1+2^\beta(1+\beta)y^\beta}\right)\\
&\le 1,
\end{split} 
\end{equation}
since $b_1\ge 1+\alpha \ge 1+\beta$. Consequently, $|(N_\beta \phi)'(x)| \le \frac{b_1}{x}N_\beta \phi(x)$ for $0<x\le 1$.

Next (see~\cite[p.870]{BT}),
\[
|(N_\beta \phi)''(x)| \le \frac{b_2}{x^2}N_\beta \phi(x)\frac{T_\beta(y)^2}{y^2(T_\beta'(y))^2}\left (1+\frac{2^\beta \beta y^\beta}{1+2^\beta (\beta+1)y^\beta} \frac{A}{b_2} \right ),
\]
where
\[
A:=3b_1(\beta+1)+(1-\beta^2)+3\frac{2^\beta (\beta+1)^2\beta y^\beta}{1+2^\beta (\beta+1)y^\beta}.
\]
Observe that $A\le 3b_1(\alpha+1)+1+3\beta (\beta+1)\le 3b_1(\alpha+1)+7$. Hence, if $b_2\ge 3b_1(\alpha+1)+7$ we have that 
\[
|(N_\beta \phi)''(x)| \le \frac{b_2}{x^2}N_\beta \phi(x) \frac{T_\beta(y)^2}{y^2(T_\beta'(y))^2} \left (1+\frac{2^\beta \beta y^\beta}{1+2^\beta (\beta+1)y^\beta} \right).
\]
On the other hand, 
\[
\frac{T_\beta(y)^2}{y^2(T_\beta'(y))^2}=\left (1-\frac{2^\beta \beta y^\beta}{1+2^\beta (\beta+1)y^\beta} \right )^2,
\]
which implies that $|(N_\beta \phi)''(x)| \le \frac{b_2}{x^2}N_\beta \phi(x)$ for $0<x\le 1$. We conclude $N_\beta \phi \in \mathcal C_2$, and thus $N_\beta \mathcal C_2 \subset \mathcal C_2$.  One can in a similar manner treat $(N_\beta \phi)^{(3)}$, and establish the second inclusion~\eqref{eq:conecontractionI} for $j=3$.

\end{proof}

\subsection{Existence and uniqueness of the random a.c.i.m}\label{racim}
Let us begin by introducing a class of random dynamics which is going to be studied in the present paper. 

Let $(\Omega, \mathcal F, \mathbb P)$ be a probability space equipped with an invertible $\mathbb P$-preserving measurable transformation $\sigma \colon \Omega \to \Omega$, and suppose that $\mathbb P$ is ergodic. 
Fix a measurable map $\beta \colon \Omega \to (0, 1)$ such that
\begin{equation}\label{alphac}
\alpha:=\esssup_{\omega \in \Omega} \beta(\omega)<1,
\end{equation}
and let $T_\omega$ be the LSV map with parameter $\beta(\omega)$. Finally, we assume that 
\[
\underline{\alpha}:=\essinf_{\omega \in \Omega}\beta(\omega)>0.
\]
By $\mathcal L_\omega$, we denote the transfer operator associated with $T_\omega$. For $\omega \in \Omega$ and $n\in \N$, set
\[
\mathcal L_\omega^n:=\mathcal L_{\sigma^{n-1}\omega}\circ \ldots \circ \mathcal L_\omega,
\]
and let $\mathcal L_\omega^0$ be the identity operator.
It follows from Lemma~\ref{lem2} and Proposition~\ref{prop:conecontraction} that there exist parameters $a$ and $b_j$ for $j\in \{1, 2, 3\}$ depending only on $\alpha$ such that for $\mathbb P$-a.e. $\omega \in \Omega$, $\L_\omega$ preserves cones $\mathcal C_*(a)$ and $\mathcal C_j$ for $j\in \{2,3\}$, where we continue to write $\C_2=\C_2(b_1, b_2)$ and $\C_3=\C_3(b_1, b_2, b_3)$.
 
 Let $\tau \colon \Omega \times [0, 1]\to \Omega \times [0, 1]$ be the associated skew-product transformation given by
 \[
 \tau(\omega, x)=(\sigma \omega, T_\omega(x)), \quad (\omega, x)\in \Omega \times [0, 1].
 \]
The following result follows from~\cite[Theorem 1.1]{KL}.
 \begin{theorem}\label{DEC}
 Let $\Omega'\subset \Omega$ be a $\sigma$-invariant set of full measure with the property that $\beta(\omega)\le \alpha$ for $\omega \in \Omega'$.
There exists $C_\alpha>0$ with the property that for $\omega \in \Omega'$,  $n\in \N$, and $\phi, \psi \in \C_*(a)$ such that $m(\phi)=m(\psi)$, we have that 
\begin{equation}\label{dec}
\int_0^1 | \L_\omega^n (\phi- \psi)|\, dm\le C_\alpha (\|\phi \|_{L^1(m)}+\|\psi \|_{L^1(m)})n^{-1/\alpha+1}.
\end{equation}
 \end{theorem}
\begin{remark}
We note that a similar conclusion follows from~\cite[Theorem 1.6]{AHNTV}  but with the additional factor $(\log n)^{1/\alpha}$ on the right-hand side in~\eqref{dec}.
\end{remark}
The following result gives the existence and uniqueness of the random a.c.i.m.
	\begin{proposition}\label{p1}
Assume that~\eqref{alphac} holds and let $\Omega'$ be as in the statement of Theorem~\ref{DEC}.
		Then, the following holds:
        \begin{enumerate}
        \item \label{p1-1}
        there exists a unique measurable map $h\colon \Omega' \to \C_{\ast}(a)\cap\C_2$ such that 
		\[
		\mathcal L_\omega h(\omega)=h(\sigma \omega) \quad \text{and} \quad \int_0^1h(\omega)\, dm=1, \quad \text{for $\omega \in \Omega'$;}
		\]
        \item \label{p1-2}
        the measure $\mu$ on $\Omega \times [0, 1]$ given by 
        \[
        \mu(A\times B)=\int_{A\cap \Omega'} \int_Bh(\omega)\, dm\,  d\mathbb P(\omega) \quad \text{for $A\in \mathcal F$ and $B\subset [0, 1]$ Borel}
        \]
        is ergodic for $\tau$. Moreover, $\mu$ is equivalent to $\mathbb P\times m$;
        \item \label{p1-3}
        $\mu$ is the unique  invariant measure for $\tau$ which is absolutely continuous with respect to $\mathbb P\times m$.
        \end{enumerate}
	\end{proposition}

    \begin{remark}
We observe that we can take $\Omega'=\bigcap_{n\in \Z}\sigma^{-n}(\bar \Omega)$, where $\bar \Omega\subset \Omega$ is a full measure set such that $\beta(\omega)\le \alpha$ for $\omega \in \bar \Omega$. The role of $\Omega'$ will be clarified in Remark~\ref{lrremark}.
    \end{remark}
	
	\begin{proof}[Proof of Proposition~\ref{p1}]
  
 Let us show claim \ref{p1-1} first.  
 We take $\omega \in \Omega'$, the constant function $1\in \mathcal C_*(a)\cap \mathcal C_3$ and consider the sequence of functions $(\psi_n^\omega)_{n\in \N}$, where $\psi_n^\omega:=\mathcal L_{\sigma^{-n}\omega}^n 1$. Observe that it follows from~\eqref{dec} that for $m, n\in \N$ with $m>n$ we have
 \[
 \|\psi_m^\omega-\psi_n^\omega\|_{L^1(m)}=\|\mathcal L_{\sigma^{-n}\omega}^n(\mathcal L_{\sigma^{-m}\omega}^{m-n}1-1)\|_{L^1(m)}\le 2C_\alpha n^{-1/\alpha+1},
 \]
 since  $\mathcal L_{\sigma^{-m}\omega}^{m-n}1\in \C_\ast(a)$ and
 $\|\mathcal L_{\sigma^{-m}\omega}^{m-n}1\|_{L^1(m)}=1$.
 Thus, the sequence $(\psi_n^\omega)_{n\in \N}$ is a Cauchy sequence (and consequently a convergent one) in $L^1(m)$. 
Set
\begin{equation}\label{homega}
h(\omega):=\lim_{n\to \infty}\psi_n^\omega=\lim_{n\to \infty}\mathcal L_{\sigma^{-n}\omega}^n (1)\in L^1(m).
\end{equation}
Clearly, $\int_0^1 h(\omega)=1$, and, moreover,
\[
h(\sigma \omega)=\lim_{n\to \infty}\mathcal L_{\sigma^{-(n-1)}\omega}^n 1=\mathcal L_\omega (\lim_{n\to \infty}\mathcal L_{\sigma^{-(n-1)}\omega}^{n-1}(1))=\mathcal L_\omega h(\omega).
\]

Next, we show that $h(\omega)\in \C_\ast(a) \cap \C_2$. To this end, we start by noting (see Remark~\ref{n06}) that since $\psi_n^\omega \in \C_\ast(a)$, we have
 \[
 0\le x^{\alpha+1}\psi_n^\omega(x)\le ax \le a, \quad \forall x\in (0, 1].
 \]
 Moreover, since $\psi_n^\omega\in \mathcal C_2$ we have that 
 \[
 \begin{split}
|(X^{\alpha+1}\psi_n^\omega)'(x)| &\le (\alpha+1)x^\alpha \psi_n^\omega(x)+
x^{\alpha+1}|(\psi_n^\omega)'(x)| \\
&\le a(\alpha+1)+b_1x^\alpha \psi_n^\omega(x) \\
&\le a(\alpha+1)+ab_1,
 \end{split}
 \]
 for all $x\in (0, 1]$. 
 Hence,  the sequence $(X^{\alpha+1}\psi_n^\omega)_{n\in \N}$ consists of equibounded and equicontinuous functions. 
 
 On the other hand, we also observe that 
  \[
  x^{\alpha+2}|(\psi_n^\omega)'(x)| \le x^{\alpha+1}|(\psi_n^\omega)'(x)|\le ab_1, \quad \forall x\in (0,1].
  \]
  Moreover,
  \[
  \begin{split}
  |(X^{\alpha+2}(\psi_n^\omega)')'(x)|&\le (\alpha+2)x^{\alpha+1}|(\psi_n^\omega)'(x)|+x^{\alpha+2}|(\psi_n^\omega)''(x)| \\
  &\le b_1(\alpha+2)x^\alpha \psi_n^\omega(x)+b_2x^\alpha \psi_n^\omega(x)\\
  &\le ab_1(\alpha+2)+ab_2,
  \end{split}
  \]
  for $x\in (0, 1]$. Hence, $(X^{\alpha+2}(\psi_n^\omega)')_{n\in \N}$ is an equibounded and equicontinuous sequence of functions.

  Similarly, 
  \[
 x^{\alpha+3}|(\psi_n^\omega)''(x)|\le x^{\alpha+2} |(\psi_n^\omega)''(x)| \le ab_2, \quad \forall x\in (0, 1].
  \]
In addition, since $\psi_n^\omega \in \mathcal C_3$ we have that 
\[
\begin{split}
|(X^{\alpha+3}(\psi_n^\omega)'')'(x)|&\le (\alpha+3)x^{\alpha+2}|(\psi_n^\omega)''(x)|+x^{\alpha+3}|(\psi_n^\omega)^{(3)}(x)|\\
&\le b_2(\alpha+3)x^\alpha \psi_n^\omega(x)+b_3x^\alpha \psi_n^\omega(x)\\
&\le ab_2(\alpha+3)+ab_3,
\end{split}
\]
for $x\in (0, 1]$. Therefore, $(X^{\alpha+3}(\psi_n^\omega)'')_{n\in \N}$ is also an equibounded and equicontinuous sequence of functions.

 By the Arzela-Ascoli theorem, 
 we can find a subsequence $(n_l)_l$ of $\N$ such that the sequences $(X^{\alpha+1}\psi_{n_l}^{\omega})_l$, $(X^{\alpha+2}(\psi_{n_l}^{\omega})')_l$ and $(X^{\alpha+3}(\psi_{n_l}^{\omega})'')_l$  converge uniformly on every compact subinterval of $(0,1]$, and thus also pointwise, to a continuous function.
 Let $\tilde h(\omega)$ denote the limit of $(X^{\alpha+1}\psi_{n_l}^{\omega})_l$.
 Then, $\psi_{n_l}^{\omega}\to X^{-\alpha-1}\tilde h( \omega)$ when $l\to \infty$ uniformly on every compact subinterval of $(0, 1]$, and in particular pointwise. Since $0\le \psi_{n_l}^{\omega}(x) \leq a x^{-\alpha}$, by the dominated convergence theorem, $\psi_{n_l}^{\omega}\to X^{-\alpha-1} \tilde h( \omega)$ as $l\to \infty$ in $L^1(m)$.  Consequently, we have (see~\eqref{homega}) that $h(\omega)=X^{-\alpha-1}\tilde h(\omega)$.
 From here it is straightforward to verify (using the fact that $\C_\ast(a)$ is closed) that $h(\omega)\in \mathcal C_*(a)$. In addition, the convergence of $(X^{\alpha+2}(\psi_{n_l}^\omega)')_l$ and $(X^{\alpha+3}(\psi_{n_l}^\omega)'')_l$ implies that the sequences $((\psi_{n_l}^k)')_l$ and $((\psi_{n_l}^k)'')_l$ converge uniformly on each compact subinterval of $(0, 1]$. We conclude that $h(\omega)$
 is of class $C^2$ 
 and that $(\psi_{n_l}^\omega)'\to h(\omega)'$ and $(\psi_{n_l}^\omega)'' \to h(\omega)''$ pointwise as $l\to \infty$. Now we can easily show that $h(\omega)\in   \mathcal C_2$, using the closedness of $\C_2$.

 Next, we discuss the measurability of the map $\omega \mapsto h(\omega)$. The arguments of \cite[Section 3.3]{GTQuas}, which remain applicable in the setting of intermittent maps, show that the map 
    $\omega \mapsto \mathcal L_\omega$ is strongly measurable,    
    when considered as a function from $\Omega'$ to the space of all bounded linear operators on a fractional Sobolev space $\mathcal H_p^t$, such that $0<t<\min \{\underline{\alpha}, 1/p \}<1$. Since the embedding $\mathcal H_p^t \hookrightarrow L^1$ is continuous, we conclude that 
    $h:\Omega'\to L^1$ is measurable, as it is the limit of measurable functions.

    Finally, we discuss the uniqueness. Suppose that $\bar h\colon \Omega' \to \C_\ast(a)\cap \C_2$ is another measurable map such that 
    \[
    \mathcal L_\omega \bar h(\omega)=\bar h(\sigma \omega) \quad \text{and} \quad \int_0^1 \bar h(\omega)\, dm=1, \quad \text{for  $\omega \in \Omega'$.}
    \]
 Then, using~\eqref{dec} one has 
 \[
 \|h(\omega)-\bar h(\omega)\|_{L^1(m)}=\|\L_{\sigma^{-n}\omega}^n(h(\sigma^{-n}\omega)-\bar h(\sigma^{-n}\omega))\|_{L^1(m)}\le 2C_\alpha n^{-1/\alpha+1},
 \]
	for $\omega \in \Omega'$ and $n\in \N$. Letting $n\to \infty$, we conclude that $h(\omega)=\bar h(\omega)$ for $\omega \in \Omega'$. This completes the proof of the first assertion.

 Next, let us show claim \ref{p1-2}. Clearly, $\mu$ is absolutely continuous with respect to $\mathbb P\times m$. Let $c>0$ and $N\in \N$ be given by Lemma~\ref{TL}. Then, 
\[
h(\omega)=\mathcal L_{\sigma^{-N}\omega}^N h(\sigma^{-N}\omega)\ge c,
\]
for  $\omega \in \Omega'$ as $h(\sigma^{-N}\omega)\in \C_\ast(a)$ and $\int_0^1 h(\sigma^{-N}\omega)\, dm=1$. This implies that $\mathbb P\times m$ is absolutely continuous with respect to $\mu$.

We now aim to show that $\mu$ is ergodic for $\tau$. Our arguments will follow closely those in the proof of~\cite[Proposition 7]{DH}.
By arguing as in the proof of~\cite[Lemma 3.4.]{nicol2018central} we have that there exists $D_\alpha>0$ such that 
\begin{equation}\label{1116}
    \|\mathcal L_\omega^n(\varphi \psi)\|_{L^1(m)}\le D_\alpha n^{-1/\alpha+1}\|\varphi\|_{C^1}\cdot \|\psi\|_{L^1(m)},
\end{equation}
for $\omega \in \Omega'$, $\varphi \in C^1 [0, 1]$ and $\psi \in \C_\ast(a)$ with $\int_0^1 \varphi \psi\, dm=0$. For $\omega \in \Omega'$, set 
\[
L_\omega (\varphi)=\frac{\mathcal L_\omega (\varphi h(\omega))}{h(\sigma \omega)}, \quad \varphi \in L^1(m).
\] 
Moreover, let
\[
L_\omega^n:=L_{\sigma^{n-1}\omega}\circ \ldots \circ L_\omega.
\]
It follows from~\eqref{1116} that 
\[
\|L_\omega^n \varphi\|_{L^1(\mu_{\sigma^n\omega})}\le C_\alpha n^{-1/\alpha+1}\|\varphi\|_{C^1} \quad \text{for $\omega\in \Omega'$ and $\varphi\in C^1[0, 1]$ with $\int_0^1 \varphi\, d\mu_\omega=0$,}
\]
where $d\mu_\omega=h(\omega)\, dm$.

Next, we claim that for  $\omega \in \Omega'$, $\varphi \in C^1[0, 1]$ and $\psi \in L^\infty(\mu_{\sigma^n \omega})$,
\begin{equation}\label{decc}
\left |\int_0^1 \varphi (\psi \circ T_\omega^n)\, d\mu_\omega-\int_0^1 \varphi\, d\mu_\omega \int_0^1 \psi\, d\mu_{\sigma^n \omega}\right | \le D_\alpha n^{-1/\alpha+1}\|\varphi\|_{C^1}\cdot \|\psi\|_{L^\infty(\mu_{\sigma^n \omega)}}.
\end{equation}
Indeed, we have \[
\begin{split}
&\left |\int_0^1 \varphi (\psi \circ T_\omega^n)\, d\mu_\omega-\int_0^1 \varphi\, d\mu_\omega \int_0^1 \psi\, d\mu_{\sigma^n \omega}\right |\\
&=\left |\int_0^1 L_\omega^n(\varphi)\psi\, d\mu_{\sigma^n\omega}-\int_0^1 \varphi\, d\mu_\omega \int_0^1 \psi\, d\mu_{\sigma^n \omega}\right |\\
&=\left |\int_0^1 L_\omega^n \left(\varphi-\int_0^1 \varphi\, d\mu_\omega \right)\psi\, d\mu_{\sigma^n\omega} \right | \\
&\le \left \|L_\omega^n \left(\varphi-\int_0^1 \varphi\, d\mu_\omega \right)\right \|_{L^1(\mu_{\sigma^n \omega)}}\cdot \|\psi\|_{L^\infty(\mu_{\sigma^n \omega)}}\\
&\le D_\alpha n^{-1/\alpha+1}\left \|\varphi-\int_0^1 \varphi\, d\mu_\omega\right \|_{C^1}\cdot \|\psi\|_{L^\infty(\mu_{\sigma^n \omega)}}\\
&\le 2D_\alpha n^{-1/\alpha+1}\|\varphi\|_{C^1}\cdot \|\psi\|_{L^\infty(\mu_{\sigma^n \omega)}},
\end{split}
\]
yielding~\eqref{decc}.

Now take measurable $\mathcal S\subset \Omega \times [0, 1]$ such that $\tau^{-1}(\mathcal S)=\mathcal S$. We need to show that $\mu (\mathcal S)\in \{0, 1\}$. For $\omega \in \Omega$, let
\[
\mathcal S_\omega:=\{x\in [0, 1]: (\omega, x)\in \mathcal S\}.
\]
Observe that 
\[
x\in T_\omega^{-1}(\mathcal S_{\sigma \omega}) \iff T_\omega (x)\in \mathcal S_{\sigma \omega} \iff (\sigma \omega, T_\omega(x))\in \mathcal S \iff \tau(\omega, x)\in \mathcal S \iff (\omega, x)\in \mathcal S,
\]
which implies that 
\begin{equation}\label{eee}
    T_\omega^{-1}(\mathcal S_{\sigma \omega})=\mathcal S_\omega, \quad \omega \in \Omega.
\end{equation}
Set
\[
\Omega_0:=\{ \omega \in \Omega': \ \mu_\omega (\mathcal S_\omega)>0\}\in \mathcal F.
\]
By~\eqref{eee} we have $\sigma(\Omega_0)=\Omega_0$. Since $\sigma$ is ergodic, we conclude that $\mathbb P(\Omega_0)\in \{0, 1\}$.  If $\mathbb P(\Omega_0)=0$, then clearly $\mu(\mathcal S)=0$.  From now on we  suppose that $\mathbb P(\Omega_0)=1$.  
We now claim that 
\begin{equation}\label{CLAIM}
\int_{\mathcal S_\omega}\varphi d\mu_\omega=0, \quad \text{for $\omega \in \Omega_0$ and $\varphi \in C^1[0, 1]$ such that $\int_0^1 \varphi \, d\mu_\omega=0$.}
\end{equation}
Indeed, \eqref{CLAIM} follows immediately from~\eqref{decc} applied to $\psi_\omega=\mathbf 1_{\mathcal S_{\sigma^n \omega}}$ for $n\in \mathbb N$. Since every $C^1$ function can be approximated by a continuous one (in supremum norm), we find that~\eqref{CLAIM} holds with $\varphi \in C^0[0, 1]$.  Finally, $C^0[0, 1]$ is dense in $L^1(\mu_\omega)$, which yields that~\eqref{CLAIM} holds for $\varphi \in L^1(\mu_\omega)$. Thus, $\mu_\omega(\mathcal S_\omega)=1$ for $\omega \in \Omega_0$, and consequently $\mu(\mathcal S)=1$.  We conclude that $\mu$ is ergodic.

Finally,  claim \ref{p1-3}
follows directly from claim \ref{p1-2}. 
	\end{proof}

    \begin{remark}
In the sequel, we will identify $\mu$ with the map $h\colon \Omega' \to \C_\ast(a)\cap \C_2$ or its arbitrary extension to a measurable map $\Omega \to \C_\ast(a)\cap \C_2$. Moreover, instead of $h$ we will write $(h(\omega))_{\omega \in \Omega}$.
\end{remark}

\section{Quenched statistical stability}
\label{S:SeqStatStab}

Before formulating a random statistical stability result, Theorem~\ref{SS}, we establish an auxiliary result.
\begin{lemma}\label{estim}
For $0<\gamma_0 \le \alpha<1$, there exist $0<\delta \le \alpha$  and $C_1, C_2>0$ such that 
\begin{equation}\label{lemm}
|(X_{\gamma}  N_{\gamma}(\phi))'(x)| \le C_1ax^{-\delta} \quad \text{and} \quad 
|(X_{\gamma}  N_{\gamma}(\phi))''(x)| \le C_2a x^{-\delta -1},
\end{equation}
for $x\in (0, 1]$, $\gamma \in [\gamma_0, \alpha]$ and $\phi \in \mathcal C_*(a)\cap \mathcal C_2(b_1, b_2)$ with $m(\phi)=1$, where $a$ and $b_i$, $i=1, 2$ are given by Lemma~\ref{lem2} and Proposition~\ref{prop:conecontraction} respectively. Moreover, there exist $\psi_i \in \mathcal C_*(a)\cap C^1(0, 1]$, $i\in \{1, 2\}$ such that 
\begin{equation}\label{psi12}
(X_\gamma N_\gamma (\phi))'=\psi_1-\psi_2,
\end{equation}
with $\int_0^1 \psi_1 \, dm = \int_0^1 \psi_2\, dm$
and $\| \psi_i\|_{L^1(m)}\le D$, $i \in \{1, 2\}$ for some $D= D_{\gamma_0, \alpha}>0$.
\end{lemma}

\begin{proof}
Since $\phi \in \mathcal C_*(a)$ and $m(\phi)=1$, we have (see~Remark~\ref{n06}) that
$0\le \phi(x)\le ax^{-\alpha}$  for $x\in (0, 1]$.  By \eqref{eq:trivialboundong}, one gets that 
\begin{equation}\label{i1}
0\le N_\gamma (\phi)(x)\le ag_\gamma(x)^{-\alpha}\le 3ax^{-\alpha}, \quad \forall x\in (0, 1].
\end{equation}
On the other hand, since $\phi \in\C_2(b_1, b_2)$, we have that $N_\gamma (\phi)\in\C_2(b_1, b_2)$, and therefore
\begin{equation}\label{i2}
|(N_\gamma(\phi))'(x)|\le \frac{b_1}{x}N_\gamma(\phi)(x)\le 3ab_1x^{-1-\alpha}, \quad \forall x\in (0, 1].
\end{equation}
Hence, using \cite[(2.3)-(2.4)]{BT}, we get that 
\begin{equation}\label{tuz}
|(X_{\gamma}  N_{\gamma}(\phi))'(x)| \le C_1ax^{\gamma-\alpha}(1-\log x) \quad x\in (0, 1],
\end{equation}
for some constant $C_1$ independent on $x$, $\phi$ and $\gamma$.
Moreover, we have that 
\[
|(N_\gamma(\phi))''(x)|\le \frac{b_2}{x^2}N_\gamma(\phi)(x)\le 3ab_2x^{-2-\alpha}, \quad \forall x\in (0, 1].
\]
Using~\cite[(2.3)-(2.5)]{BT}, we obtain that 
\begin{equation}\label{i3}
|(X_{\gamma}  N_{\gamma}(\phi))''(x)| \le C_2ax^{\gamma-\alpha-1}(1-\log x) \quad x\in (0, 1],
\end{equation}
for some constant $C_2$ independent on $x$, $\phi$ and $\gamma$. Choose now $r>0$ such that $\delta:=\alpha+r-\gamma_0 \le \alpha$. Since $\lim\limits_{x\to 0^+}\frac{1-\log x}{x^{-r}}=0$,
it follows from~\eqref{tuz} that  there exists $\bar C_1>0$ (independent on $x$, $\phi$ and $\gamma$) such that 
\[
|(X_{\gamma}  N_{\gamma}(\phi))'(x)| \le \bar  C_1a x^{\gamma-\alpha-r}\le \bar C_1ax^{\gamma_0-\alpha-r}= \bar C_1ax^{-\delta} \quad x\in (0, 1],
\]
which yields the first estimate in~\eqref{lemm}. Similarly, one can establish the second estimate in~\eqref{lemm}. 

In order to establish the second assertion of the lemma, we first note that  $m((X_\gamma N_\gamma (\phi))')=0$, using integration by parts.  It follows from~\eqref{lemm} and Lemma~\ref{insidecone} that  there exists $\lambda>0$ depending only on $\alpha$ such that 
$(X_\gamma N_\gamma (\phi))'+\lambda x^{-\alpha} \in \mathcal C_\ast(a)$. Set
\[
\psi_1:=(X_\gamma N_\gamma (\phi))'+\lambda x^{-\alpha} \quad \text{and} \quad \psi_2:=\lambda x^{-\alpha}.
\]
Clearly, \eqref{psi12} holds and $\psi_i \in \mathcal C_\ast(a)\cap C^1(0, 1]$ for $i=1, 2$ (see Remark~\ref{181}). Finally, using the first estimate in~\eqref{lemm} we have that 
\[
m(\psi_1)\le \frac{C_1a+\lambda}{1-\alpha} \quad \text{and} \quad m(\psi_2)\le \frac{\lambda}{1-\alpha}.
\]
Hence, we can take $D:=\frac{C_1a+\lambda}{1-\alpha}>0$.
\end{proof}

Let $\beta_i \colon \Omega \to (0, 1)$, $i=1, 2$   be measurable maps such that
 \[
0< \essinf_{\omega \in \Omega} \beta_i(\omega) \quad \text{and} \quad \esssup_{\omega \in \Omega} \beta_i(\omega)<\frac 1 2,
 \]
 for $i=1, 2$. 
By $(h_i(\omega))_{\omega \in \Omega}$, $i=1, 2$ 
we will denote the random a.c.i.m corresponding to the cocycle $(T_{\beta_i(\omega)})_{\omega \in \Omega}$, 
identified in Proposition~\ref{p1}.
\begin{theorem}\label{SS} 
There exists a constant $C_\alpha>0$ such that for each $\psi \in L^\infty(m)$ and $\mathbb P$-a.e. $\omega \in \Omega$,
\begin{equation}\label{eq:Lipbound}
	\left|\int_0^1\psi\left(h_1(\omega)-h_2(\omega)\right)\, dm\right|\le C_\alpha\epsilon \|\psi\|_{\infty},
	\end{equation}
    where $\epsilon:=\esssup_{\omega \in \Omega}|\beta_1(\omega)-\beta_2(\omega)|$.
    Consequently,
 \begin{equation}\label{ssl1}
\|h_1(\omega)-h_2(\omega)\|_{L^1} \le C_\alpha\epsilon, \quad \text{for $\mathbb P$-a.e. $\omega \in \Omega$.}
 \end{equation}
\end{theorem}
\begin{proof}
Observe that the conclusion of the theorem holds trivially in the case when $\epsilon=0$. Therefore, in the sequel, we assume that $\epsilon>0$.
By $\mathcal L_{\omega, i}$, we denote the transfer operator associated with $T_{\beta_i(\omega)}$ for $\omega \in \Omega$ and $i=1, 2$.
We start by observing that for $\mathbb P$-a.e. $\omega \in \Omega$ and every $n\in \N$,
\[
    \begin{split}
		&\int_0^1\psi\left(h_1(\omega)-h_2(\omega)\right)\,dm \\
        &=\int_0^1\psi(\L_{\sigma^{-n}\omega, 1}^n -\L_{\sigma^{-n}\omega, 2}^n)h_2(\sigma^{-n}\omega)\,dm+\int_0^1\psi\L_{\sigma^{-n}\omega, 1}^n (h_1(\sigma^{-n}\omega)-h_2(\sigma^{-n}\omega))\,dm \\
  &=:(I)_{\omega, n}+(II)_{\omega, n}.
	\end{split}
 \]
 Since $h_i(\omega)\in \C_\ast(a)$, $i=1, 2$ have integral one, Theorem~\ref{DEC} implies that 
 \begin{align*}
\left|\int_0^1\psi \L_{\sigma^{-n}\omega, 1}^n (h_1(\sigma^{-n}\omega)-h_2(\sigma^{-n}\omega))\, dm\right|\le C_\alpha\|\psi\|_\infty n^{-1/\alpha+1},
	\end{align*}
    for $\mathbb P$-a.e. $\omega \in \Omega$ and $n\in \N$. 
 By choosing $n$ sufficiently large, we obtain that 
 \begin{equation}\label{eq:boundII}
	|(II)_{\omega, n}|\le C_\alpha \epsilon \|\psi \|_\infty, \quad \text{for $\mathbb P$-a.e. $\omega \in \Omega$.}
\end{equation}
On the other hand, we have that
\begin{align*}
(I)_{\omega, n}&=\int_0^1\psi\sum_{j=1}^n\L_{\sigma^{j-n}\omega, 1}^{n-j}(\L_{\sigma^{j-1-n}\omega, 1}-\L_{\sigma^{j-1-n}\omega, 2})\L_{\sigma^{-n}\omega, 2}^{j-1}h_2(\sigma^{-n}\omega)\,dm\\
&=\sum_{j=1}^n\int_0^1\psi \L_{\sigma^{j-n}\omega, 1}^{n-j}\int_{\beta_1(\sigma^{j-1-n}\omega)}^{\beta_2(\sigma^{j-1-n}\omega)}\partial_{\gamma}\L_{\gamma}(h_2(\sigma^{j-1-n}\omega))\,d\gamma \, dm.
\end{align*}
Since (see~\cite[p.865]{BT}) $
\partial_\gamma\L_\gamma(\phi)=-(X_\gamma  N_\gamma(\phi))'
$, one has that 
\begin{align*}
(I)_{\omega, n}=-\sum_{j=1}^n\int_0^1\int_{\beta_1(\sigma^{j-1-n}\omega)}^{\beta_2(\sigma^{j-1-n}\omega)}\psi \L_{\sigma^{j-n}\omega, 1}^{n-j}(X_{\gamma}  N_{\gamma}(h_2(\sigma^{j-n-1}\omega)))'\,  d\gamma\, dm.
\end{align*}

Since 
$h_2(\sigma^{j-1-n}\omega)\in \mathcal C_*(a)\cap \mathcal C_2(b_1, b_2)$ and $\int_0^1h_2(\sigma^{j-1-n}\omega)\, dm=1$,   it follows from Lemma~\ref{estim} that  $(X_{\gamma}  N_{\gamma}(h_2(\sigma^{j-1-n}\omega)))'$ has zero integral and it is a difference of two functions in $\mathcal C_\ast (a)$ whose $L^1(m)$-norms are bounded by some constant independent on $\gamma$, $j$ and $\omega$.
 Hence, \eqref{dec}  implies
 that 
\[
\begin{split}
&\left |\int_0^1\int_{\beta_1(\sigma^{j-1-n}\omega)}^{\beta_2(\sigma^{j-1-n}\omega)}\psi \L_{\sigma^{j-n}\omega, 1}^{n-j}(X_{\gamma}  N_{\gamma}(h_2(\sigma^{j-n-1}\omega)))'\,  d\gamma\, dm \right | \\
&\le \int_{\beta_1(\sigma^{j-1-n}\omega)}^{\beta_2(\sigma^{j-1-n}\omega)}\|\psi \L_{\sigma^{j-n}\omega, 1}^{n-j}(X_{\gamma}  N_{\gamma}(h_2(\sigma^{j-n-1}\omega)))'\|_{L^1(m)}\, d\gamma \\
&\le \| \psi \|_{L^\infty(m)}\int_{\beta_1(\sigma^{j-1-n}\omega)}^{\beta_2(\sigma^{j-1-n}\omega)}\|\L_{\sigma^{j-n}\omega, 1}^{n-j}(X_{\gamma}  N_{\gamma}(h_2(\sigma^{j-n-1}\omega)))'\|_{L^1(m)}\, d\gamma  \\
&\le C_\alpha |\beta_2(\sigma^{j-1-n}\omega)-\beta_1(\sigma^{j-1-n}\omega)| \dfrac{1}{(n-j)^{1/\alpha-1}} \|\psi\|_{L^\infty} \\
&\le C_\alpha \epsilon\dfrac{1}{(n-j)^{1/\alpha-1}}\|\psi\|_{L^\infty},
\end{split}
\]
with the convention that 
$\frac{1}{(n-j)^{1/\alpha-1}}:=1$ for $j=n$.
 Hence, for $\mathbb P$-a.e. $\omega \in \Omega$,
\begin{equation}\label{eq:boundI}
|(I)_{\omega,n}| \le C_\alpha \epsilon \| \psi\|_{L^\infty}\sum_{j=1}^n  \dfrac{1}{(n-j)^{1/\alpha-1}} \le C_\alpha \epsilon \|\psi\|_{L^\infty},
\end{equation}
since the series  $\sum_{j\ge 1} \dfrac{1}{j^{1/\alpha-1}}$ converges. The conclusion of the proposition now follows readily from~\eqref{eq:boundII} and \eqref{eq:boundI}.

\end{proof}

	\section{Quenched linear response}\label{S:LinRep}
 In this section, we establish a quenched linear response for families of perturbed intermittent random systems.
 Let $0<\underline{\alpha}<\alpha <\frac 1 2$ and $\epsilon_0>0$. 
Take $\beta\colon \Omega\to (0, 1)$, $\delta \colon \Omega \to [0, 1)$   measurable maps such that
\begin{equation}\label{eq:randomB}
\underline{\alpha}\le \essinf_{\omega \in \Omega}\beta(\omega)-\epsilon_0 \quad \text{and} \quad \esssup_{\omega \in \Omega}\beta(\omega)+\epsilon_0\le \alpha.
\end{equation}
By $(h_\varepsilon(\omega))_{\omega \in \Omega}$ we will denote the random a.c.i.m associated with the cocycle \\$(T_{\beta(\omega)+\epsilon \delta(\omega)})_{\omega \in \Omega}$ for $\epsilon \in (-\epsilon_0, \epsilon_0)$. It follows from Proposition~\ref{p1} and its proof that $h_\epsilon(\omega)$ is uniquely determined for $\omega \in \Omega'$ and each $\epsilon \in (-\epsilon_0, \epsilon_0)$, where $\Omega'\subset \Omega$ is a $\sigma$-invariant set of full measure such that 
\[
\underline{\alpha}\le \beta(\omega)-\epsilon_0 \quad \text{and} \quad \beta(\omega)+\epsilon_0 \le \alpha,
\]
for each $\omega \in \Omega'$. In the sequel, we will also write $h(\omega)$ instead of $h_0(\omega)$.

 The main result is the following.
                                \begin{theorem}
                                \label{thm:LR}
                                For  $\omega \in \Omega'$ and $\psi\in L^\infty(m)$,
                \begin{equation}\label{eq:LR}
                \begin{split}
               & \lim_{\epsilon\to0}\frac{1}{\epsilon} \int_0^1\psi\left(h_\epsilon(\omega)-h(\omega)\right)dm\\
               &=
                - \sum_{i=0}^\infty \delta(\sigma^{-(i+1)}\omega) \int_0^1\psi \, \L_{\sigma^{-i}\omega}^i \left(X_{\beta(\sigma^{-(i+1)}\omega)} N_{\beta(\sigma^{-(i+1)}\omega)} (h(\sigma^{-(i+1)}\omega))\right)'\, dm,
                \end{split}
                \end{equation}
         where $\L_\omega$ is the transfer operator associated with $T_{\beta(\omega)}$.
                \end{theorem}
\begin{remark}\label{lrremark}
Consider a parametrized family $(T_{\omega, \epsilon})_{\omega \in \Omega}$ for $\epsilon \in (-\epsilon_0, \epsilon_0)$ of random dynamical systems on a compact manifold $M$ that admit a unique physical measure $\mu_\epsilon$ which is absolutely continuous with respect to $\mathbb P\times m$, where $m$ is the Lebesgue measure on $M$. Let $H_\epsilon  \in L^1(\mathbb P\times m)$ be such that $d\mu=H_\epsilon d(\mathbb P\times m)$ and denote $h_\epsilon (\omega)=H_\epsilon (\omega, \cdot)$. In general, for each $\epsilon \in (-\epsilon_0, \epsilon_0)$, $h_\epsilon(\omega)$ is uniquely determined on a set $\Omega_\epsilon \subset \Omega$ of full measure. Consequently, the statement in the spirit of~\eqref{eq:LR} makes sense for $\omega \in \bigcap_{\epsilon}\Omega_\epsilon$. In principle, this intersection can fail to be a set of full measure or even  measurable. We refer to~\cite{dragivcevic2023quenched} for a more detailed discussion.  In our case, by Proposition~\ref{p1}, $h_\epsilon (\omega)$ is uniquely determined on a set of full measure that does not depend on $\epsilon$.

\end{remark}

Before proving Theorem~\ref{thm:LR},
 we begin with the following auxiliary result which is similar in nature to~\cite[Lemma 4.5]{L}.
\begin{lemma}\label{HZU}
For $0<\gamma_0 \le \alpha$, there exist $0<\delta \le \alpha$ and $C_1', C_2'>0$ such that 
\begin{equation}\label{p11}
|\partial_\gamma^2 \mathcal L_\gamma (\phi)(x)| \le C_1'x^{-\delta} \quad \text{and} \quad |(\partial_\gamma^2 \mathcal L_\gamma (\phi))'(x)| \le C_2'x^{-\delta-1},
\end{equation}
for $x\in (0, 1]$, $\gamma \in [\gamma_0, \alpha]$ and $\phi \in \mathcal C_*(a)\cap \mathcal C_3(b_1, b_2, b_3)$ with $m(\phi)=1$. Moreover, there exist $\psi_i \in \C_\ast(a)$, $i\in \{1, 2\}$ such that 
\[
\partial_\gamma^2 \L_\gamma (\phi)=\psi_1-\psi_2,
\]
with $\int_0^1 \psi_1 \, dm = \int_0^1 \psi_2 \, dm$
and $\| \psi_i\|_{L^1(m)}\le D$, $i \in \{1, 2\}$ for some $D>0$ which depends only on $\gamma_0$ and $\alpha$.
\end{lemma}

\begin{proof}
Throughout the proof of this lemma, $c$ will denote a generic positive constant (independent on $x$, $\phi$ and $\gamma$) that can change from one occurrence to the next.
By~\cite[Lemma 4.3.]{L}, we have that 
\begin{equation}\label{secondderiv}
\partial_\gamma^2 \mathcal L_\gamma (\phi)(x)=(\partial_\gamma X_\gamma N_\gamma (\phi))'(x)+X_\gamma'(X_\gamma N_\gamma (\phi))'(x)+X_\gamma (X_\gamma N_\gamma (\phi))''(x).
\end{equation}
Moreover, by \cite[p. 866]{BT}, there exists $c>0$ such that 
\begin{equation}\label{abu}
|\partial_\gamma X_\gamma (x)| \le cx^{1+\gamma}(1-\log x)^2 \quad \text{and} \quad |\partial_\gamma X_\gamma'(x)| \le cx^\gamma (1-\log x)^2,
\end{equation}
for $x\in (0, 1]$ and $\gamma \in [\gamma_0, \alpha]$. 
Then, we firstly have (using~\eqref{i1}, \eqref{i2} and~\eqref{abu}) that 
\[
\begin{split}
|(\partial_\gamma X_\gamma N_\gamma (\phi))'(x)| &=|(\partial_\gamma X_\gamma)'(x)N_\gamma (\phi)(x)+\partial_\gamma X_\gamma (x)(N_\gamma (\phi))'(x)|  \\
&\le |\partial_\gamma X_\gamma'(x)N_\gamma (\phi)(x)|+|\partial_\gamma X_\gamma (x)(N_\gamma (\phi))'(x)|  \\
&\le cx^\gamma (1-\log x)^2 3ax^{-\alpha}+cx^{1+\gamma}(1-\log x)^2 3ab_1x^{-1-\alpha}\\
&=c x^{\gamma-\alpha}(1-\log x)^2.
\end{split}
\]
Secondly,  using~\cite[(2.4)]{BT} and~\eqref{tuz} we obtain that  
\[
|X_\gamma'(x)(X_\gamma N_\gamma (\phi))'(x)| \le cx^{2\gamma -\alpha}(1-\log x)^2 \le cx^{\gamma-\alpha}(1-\log x)^2,
\]
for $x\in (0, 1]$.  Thirdly, \cite[(2.3)]{BT} and~\eqref{i3} give that 
\[
|X_\gamma (x) (X_\gamma N_\gamma (\phi))''(x)| \le cx^{\gamma-\alpha}(1-\log x)^2,
\]
for $x\in (0, 1]$. Hence, \eqref{secondderiv}  implies that 
\[
|\partial_\gamma^2 \mathcal L_\gamma (\phi)(x)| \le cx^{\gamma_0-\alpha}(1-\log x)^2,
\]
for $x\in (0, 1]$.  Thus, by choosing $\delta$ as in the proof of Lemma~\ref{estim}, we obtain the existence of $C_1'>0$ such that the first estimate in~\eqref{p11} holds.

In addition, differentiating \eqref{secondderiv} yields 
\[
\begin{split}
(\partial_\gamma^2 \mathcal L_\gamma (\phi))'(x) &=(\partial_\gamma  X_\gamma'')(x)N_\gamma (\phi)(x)+\partial_\gamma X_\gamma'(x)(N_\gamma (\phi))'(x)+
\partial_\gamma X_\gamma'(x)(N_\gamma (\phi))'(x)\\
&\phantom{=}+\partial_\gamma X_\gamma (x)(N_\gamma (\phi))''(x)+
X_\gamma''(x)(X_\gamma N_\gamma (\phi))'(x)+
X_\gamma'(x)(X_\gamma N_\gamma (\phi))''(x) \\
&\phantom{=}+X_\gamma'(x)(X_\gamma N_\gamma (\phi))''(x)+X_\gamma(x)(X_\gamma N_\gamma (\phi))^{(3)}(x).
\end{split}
\]
Since (see~\cite[p.18]{L}) $|\partial_\gamma X_\gamma''(x)| \le cx^{\gamma-1}(1-\log x)^2$, we have that 
\[
|(\partial_\gamma  X_\gamma'')(x)N_\gamma (\phi)(x)| \le cx^{\gamma-\alpha-1}(1-\log x)^2, \quad x\in (0, 1].
\]
One can in an analogous manner treat all other terms (showing that they are of order $x^{\gamma-\alpha-1}(1-\log x)^2$) except for $X_\gamma^2(N_\gamma (\phi))^{(3)}$ (which comes from developing $X_\gamma (X_\gamma N_\gamma (\phi))^{(3)}$). Since $\phi \in \mathcal C_3$, we have that $N_\gamma \phi \in \mathcal C_3$, and consequently 
\[
|X_\gamma^2(x)(N_\gamma (\phi))^{(3)}(x)| \le cx^{2+2\gamma}(1-\log x)^2 x^{-3-\alpha}\le cx^{\gamma-\alpha-1}(1-\log x)^2, \quad x\in (0, 1].
\]
The second estimate in~\eqref{p11} now follows easily.

Since $m(\partial_\gamma^2 \L_\gamma (\phi))=0$, the second assertion of the lemma follows from~\eqref{p11} and Lemma~\ref{insidecone} by arguing exactly as in the proof of Lemma~\ref{estim}.
\end{proof}

\begin{remark}
The proof of Theorem~\ref{thm:LR} below is guided by the proof of~\cite[Theorem 1.1]{L}. Since we consider the case of random dynamics, adequate changes are incorporated. In addition, in our setting one lacks the sharp control on the numbers $b_p^\omega$ introduced in~\eqref{b_pt} which holds in the autonomous setting. More precisely, in the autonomous case 
in~\eqref{second} $p^{-1/\underline{\alpha}}$ can be replaced by $p^{-1/\alpha}$. 

We stress that we do not know how to obtain a version of Theorem~\ref{thm:LR} without the requirement that $\alpha<1/2$. Indeed, without this condition we are even unable to show that the series in~\eqref{eq:LR} converges. 

\end{remark}

\begin{proof}[Proof of Theorem~\ref{thm:LR}]
We first show that the series  in~\eqref{eq:LR} converges. Indeed, this easily follows from~\eqref{dec} and  Lemma~\ref{estim}  since  $h(\omega)\in \C_\ast(a)\cap \C_2(b_1, b_2)$ for $\omega \in \Omega'$.
Next, we show the equality in \eqref{eq:LR}.
Let us fix $\omega\in \Omega'$ arbitrary and set\begin{equation}\label{eq:L}
L=L(\omega):=- \sum_{i=0}^\infty \delta(\sigma^{-(i+1)}\omega)   \int_0^1 \psi \, \L_{\sigma^{-i}\omega}^i \left(X_{\beta(\sigma^{-(i+1)} \omega)}N_{\beta(\sigma^{-(i+1)}\omega)} (h(\sigma^{-(i+1)}\omega)) \right)'\, dm.
\end{equation}
Then,  for every $l\in \N$ we have 
\[
\begin{split}
&\int_0^1\psi (h_\epsilon(\omega)-h(\omega))\, dm \\
&=\int_0^1 \psi (\L_{\sigma^{-l}\omega}^\epsilon)^lh_\epsilon (\sigma^{-l}\omega) \, dm-\int_0^1\psi \L_{\sigma^{-l}\omega}^l h(\sigma^{-l}\omega)\, dm\\
&=\int_0^1  \psi (\L_{\sigma^{-l}\omega}^\epsilon)^l 1\, dm-\int_0^1\psi \L_{\sigma^{-l}\omega}^l1\, dm + \int_0^1\psi (\L_{\sigma^{-l}\omega}^\epsilon)^l(h_\epsilon(\sigma^{-l}\omega)-1)\, dm \\
&\phantom{=}+\int_0^1\psi \L_{\sigma^{-l}\omega}^l(1-h(\sigma^{-l}\omega))\, dm,
\end{split}
\]
where $\L_\omega^\epsilon$ is the transfer operator associated with $T_{\beta(\omega)+\epsilon \delta(\omega)}$, 
\[
(\mathcal L_\omega^\epsilon)^p:=\L_{\sigma^{p-1}\omega}^\epsilon\circ \ldots \circ \L_{\sigma \omega}^\epsilon \circ \L_\omega^\epsilon , \quad \omega \in \Omega', \ p\in \N \setminus \{0\},
\]
and $\L_t^0:=\Id$.
Hence, \eqref{dec} implies that
\begin{equation}\label{iu}
\int_0^1\psi (h_\epsilon(\omega)-h(\omega))\, dm=\int_0^1  \psi (\L_{\sigma^{-l}\omega}^\epsilon)^l 1\, dm-\int_0^1\psi \L_{\sigma^{-l}\omega}^l1\, dm+O(\|\psi \|_{L^\infty(m)}l^{1-1/\alpha}).
\end{equation}
Take $\xi>0$ arbitrary. By~\eqref{iu}, we have that there exists $l_0\in \N$ depending only on $\xi$ and $\epsilon \in (-\epsilon_0, \epsilon_0)\setminus \{0\}$ such that  for $l\ge l_0$,
\begin{equation}\label{eq:diffder}
\begin{split}
\frac{\int_0^1\psi (h_\epsilon(\omega)-h(\omega))\, dm-L\epsilon}{\epsilon} &=\frac{\int_0^1  \psi (\L_{\sigma^{-l}\omega}^\epsilon)^l 1\, dm-\int_0^1\psi \L_{\sigma^{-l}\omega}^l1\, dm-L\epsilon}{\epsilon}\\
&\phantom{=}+O(\| \psi \|_{L^\infty(m)}\lvert \epsilon \rvert^\xi).
\end{split}
\end{equation}
On the other hand, 
\[
\int_0^1  \psi (\L_{\sigma^{-l}\omega}^\epsilon)^l 1\, dm-\int_0^1\psi \L_{\sigma^{-l}\omega}^l1\, dm=\sum_{j=0}^{l-1}\int_0^1 \psi  (\L_{\sigma^{-j}\omega}^\epsilon )^j (\L_{\sigma^{-(j+1)}\omega}^\epsilon-\L_{\sigma^{-(j+1)}\omega})\L_{\sigma^{-l}\omega}^{l-1-j}(1)\, dm.
\]
Writing 
\[
\begin{split}
&(\L_{\sigma^{-(j+1)}\omega}^\epsilon-\L_{\sigma^{-(j+1)}\omega})(\varphi)(x)\\
&=\epsilon \delta(\sigma^{-(j+1)}\omega)\partial_\gamma \L_{\beta(\sigma^{-(j+1)}\omega)} (\varphi)(x)\\
&\phantom{=}+\epsilon^2 (\delta(\sigma^{-(j+1)}\omega))^2\int_0^1 (1-t)\partial_\gamma^2 \L_{\beta(\sigma^{-(j+1)}\omega)+t\epsilon \delta(\sigma^{-(j+1)}\omega)}(\varphi)(x)\, dt,
\end{split}
\]
 we conclude that, in order to control \eqref{eq:diffder}, we must deal with the following two expressions: 
 \begin{equation}\label{E1}
 \sum_{j=0}^{l-1}\delta(\sigma^{-(j+1)}\omega) \int_0^1  \psi (\L_{\sigma^{-j}\omega}^\epsilon )^j\partial_\gamma \L_{\beta(\sigma^{-(j+1)}\omega)} (\L_{\sigma^{-l}\omega}^{l-1-j}(1))\, dm  -L
 \end{equation}
 and 
\begin{equation}\label{E2}
\epsilon \sum_{j=0}^{l-1} (\delta (\sigma^{-(j+1)}\omega))^2 \int_0^1 \int_0^1 (1-t) \psi  (\L_{\sigma^{-j}\omega}^\epsilon )^j \partial_\gamma^2 \L_{\beta(\sigma^{-(j+1)}\omega)+t\epsilon \delta(\sigma^{-(j+1)}\omega)}(\L_{\sigma^{-l}\omega}^{l-1-j}(1))\, dm\, dt.
\end{equation}
\paragraph{\bf{Bound on \eqref{E2}}.}
Since $\L_\omega^j(1)\in \C_\ast(a)\cap \C_3(b_1, b_2, b_3)$ for $\omega\in \Omega'$ and $j\in \N$, it follows from~\eqref{dec} and Lemma~\ref{HZU} that
\[
\begin{split}
&\sum_{j=1}^{l-1} (\delta(\sigma^{-(j+1)}\omega))^2 \int_0^1 \int_0^1 (1-t) |\psi  (\L_{\sigma^{-j}\omega}^\epsilon )^j \partial_\gamma^2 \L_{\beta(\sigma^{-(j+1)}\omega)+t\epsilon \delta(\sigma^{-(j+1)}\omega)}(\L_{\sigma^{-l}\omega}^{l-1-j}(1))|\, dm\, dt  \\
&\le \sum_{j=1}^\infty \int_0^1 \int_0^1 (1-t) |\psi  (\L_{\sigma^{-j}\omega}^\epsilon )^j \partial_\gamma^2 \L_{\beta(\sigma^{-(j+1)}\omega)+t\epsilon \delta(\sigma^{-(j+1)}\omega)}(\L_{\sigma^{-l}\omega}^{l-1-j}(1))|\, dm\, dt  \\
&\le D\sum_{j=1}^\infty j^{1-1/\alpha}<+\infty,
\end{split}
\]
where $D>0$ is some constant independent on $\epsilon \in (-\epsilon_0, \epsilon_0)$.  This readily implies that~\eqref{E2} converges to $0$ when $\epsilon \to 0$.

\paragraph{\bf{Bound on \eqref{E1}}.}
Take $\eta>0$.
Since $\partial_\gamma \L_{\beta(\sigma^{-(j+1)}\omega)} (\phi)=-(X_{\beta(\sigma^{-(j+1)}\omega)}N_{\beta(\sigma^{-(j+1)}\omega)}(\phi))'$, it follows from~\eqref{dec} and Lemma~\ref{estim} that
there exists $n\in \N$ such that whenever $l$ is sufficiently large, 
\begin{equation}\label{eq:27dev}
\begin{split}
&\left | \sum_{j=0}^{l-1}\delta(\sigma^{-(j+1)}\omega)\int_0^1 \psi (\L_{\sigma^{-j}\omega}^\epsilon )^j\partial_\gamma \L_{\beta(\sigma^{-(j+1)}\omega)} (\L_{\sigma^{-l}\omega}^{l-1-j}(1))\, dm  -L \right |  \\
&=\bigg |-\sum_{j=0}^{l-1}\delta(\sigma^{-(j+1)}\omega)\int_0^1 \psi (\L_{\sigma^{-j}\omega}^\epsilon )^j (X_{\beta(\sigma^{-(j+1)}\omega)} N_{\beta(\sigma^{-(j+1)}\omega)}(\L_{\sigma^{-l}\omega}^{l-1-j}(1)))'\, dm+  \\
&\phantom{=}+\sum_{j=0}^\infty \delta(\sigma^{-(j+1)}\omega) \int_0^1 \psi \, \L_{\sigma^{-j}\omega}^j \left(X_{\beta(\sigma^{-(j+1)}\omega)} N_{\beta(\sigma^{-(j+1)}\omega)} (h(\sigma^{-(j+1)}\omega)) \right)'\, dm \bigg | \\
&\le \bigg | \sum_{j=0}^n \delta(\sigma^{-(j+1)}\omega)\int_0^1 \psi (\L_{\sigma^{-j}\omega}^\epsilon )^j (X_{\beta(\sigma^{-(j+1)}\omega)} N_{\beta(\sigma^{-(j+1)}\omega)}(\L_{\sigma^{-l}\omega}^{l-1-j}(1)))'\, dm  -\\
&\phantom{\le}-\sum_{j=0}^n \delta(\sigma^{-(j+1)}\omega)\int_0^1 \psi \, \L_{\sigma^{-j}\omega}^j \left(X_{\beta(\sigma^{-(j+1)}\omega)} N_{\beta(\sigma^{-(j+1)}\omega)} (h(\sigma^{-(j+1)}\omega)) \right)'\, dm \bigg |+\eta \\
&\le \left |\sum_{j=0}^n \delta (\sigma^{-(j+1)}\omega)\int_0^1\psi ((\L_{\sigma^{-j}\omega}^\epsilon )^j -\L_{\sigma^{-j}\omega}^j)[(X_{\beta(\sigma^{-(j+1)}\omega)} N_{\beta(\sigma^{-(j+1)}\omega)}(\L_{\sigma^{-l}\omega}^{l-1-j}(1)))']\, dm \right| \\
&\phantom{\le}+
\bigg |\sum_{j=0}^n \delta(\sigma^{-(j+1)}\omega)\int_0^1 \psi \L_{\sigma^{-j}\omega}^j((X_{\beta(\sigma^{-(j+1)}\omega)} N_{\beta(\sigma^{-(j+1)}\omega)}(\L_{\sigma^{-l}\omega}^{l-1-j}(1)))'-\\
&\phantom{\le}
-(X_{\beta(\sigma^{-(j+1)}\omega)} N_{\beta(\sigma^{-(j+1)}\omega)} (h(\sigma^{-(j+1)}\omega)))')\, dm\bigg | +\eta \\
&=:\eqref{eq:27dev}(I)+\eqref{eq:27dev}(II)+\eta.
\end{split}
\end{equation}
\paragraph{\bf{Bound on \eqref{eq:27dev}(I).}}
\[
\begin{split}
&\left |\sum_{j=0}^n \delta(\sigma^{-(j+1)}\omega) \int_0^1\psi ((\L_{\sigma^{-j}\omega}^\epsilon )^j -\L_{\sigma^{-j}\omega}^j)[(X_{\beta(\sigma^{-(j+1)}\omega)} N_{\beta(\sigma^{-(j+1)}\omega)}(\L_{\sigma^{-l}\omega}^{l-1-j}(1)))']\, dm \right| \\
&\le \| \psi \|_{L^\infty(m)}\sum_{j=1}^n \|((\L_{\sigma^{-j}\omega}^\epsilon )^j -\L_{\sigma^{-j}\omega}^j)[(X_{\beta(\sigma^{-(j+1)}\omega)} N_{\beta(\sigma^{-(j+1)}\omega)}(\L_{\sigma^{-l}\omega}^{l-1-j}(1)))']\|_{L^1(m)}. 
\end{split}
\]
By Lemma~\ref{estim}, we have that 
\[
(X_{\beta(\sigma^{-(j+1)}\omega)} N_{\beta(\sigma^{-(j+1)}\omega)}(\L_{\sigma^{-l}\omega}^{l-1-j}(1)))'= \psi_1-\psi_2,
\]
for some $\psi_i=\psi_i(j, l, \omega)\in \C_\ast(a)\cap C^1(0, 1]$, $i=1, 2$. Moreover, $\| \psi_i\|_{L^1(m)}$ is bounded from above by a constant independent on indices $j, l$ and $\omega$. Thus,
\[
\begin{split}
&((\L_{\sigma^{-j}\omega}^\epsilon )^j -\L_{\sigma^{-j}\omega}^j)[(X_{\beta(\sigma^{-(j+1)}\omega)} N_{\beta(\sigma^{-(j+1)}\omega)}(\L_{\sigma^{-l}\omega}^{l-1-j}(1)))'] \\
&=((\L_{\sigma^{-j}\omega}^\epsilon )^j -\L_{\sigma^{-j}\omega}^j) (\psi_1-\psi_2) \\
&=\sum_{r=1}^j (\L_{\sigma^{-j+r}\omega}^\epsilon)^{j-r} (\L_{\sigma^{-j+r-1}\omega}^\epsilon-\L_{\sigma^{-j+r-1}\omega})\L_{\sigma^{-j}\omega}^{r-1}(\psi_1-\psi_2) \\
&=\sum_{r=1}^j (\L_{\sigma^{-j+r}\omega}^\epsilon)^{j-r} \int_{\beta(\sigma^{-j+r-1}\omega)}^{\beta(\sigma^{-j+r-1}\omega)+\epsilon \delta(\sigma^{-j+r-1}\omega)}\partial_\gamma \L_\gamma  (\L_{\sigma^{-j}\omega}^{r-1}(\psi_1-\psi_2))\, d\gamma \\
&=\sum_{r=1}^j (\L_{\sigma^{-j+r}\omega}^\epsilon)^{j-r} \int_{\beta(\sigma^{-j+r-1}\omega)}^{\beta(\sigma^{-j+r-1}\omega)+\epsilon \delta(\sigma^{-j+r-1}\omega)}\partial_\gamma \L_\gamma  (\L_{\sigma^{-j}\omega}^{r-1}\psi_1)\, d\gamma \\
&\phantom{=}-\sum_{r=1}^j (\L_{\sigma^{-j+r}\omega}^\epsilon)^{j-r} \int_{\beta(\sigma^{-j+r-1}\omega)}^{\beta(\sigma^{-j+r-1}\omega)+\epsilon \delta(\sigma^{-j+r-1}\omega)}\partial_\gamma \L_\gamma  (\L_{\sigma^{-j}\omega}^{r-1}\psi_2)\, d\gamma. \\
\end{split}
\]
Let
\[
\bar{\psi}_i:=\L_{\sigma^{-j}\omega}^{r-1}\psi_i \in \C_\ast(a)\cap C^1(0, 1] , \quad i\in \{1, 2\}.
\]
Then, using~\cite[(2.3)-(2.4)]{BT}, Remark~\ref{n06} and noting that $\| \bar{\psi}_i\|_{L^1(m)}=\|\psi_i\|_{L^1(m)}$
we have that 
\begin{equation}\label{uwr}
\begin{split}
|\partial_\gamma \L_\gamma \bar{\psi}_i(x)| &=|(X_\gamma N_\gamma (\bar{\psi}_i))'(x)| \\
&\le |X_\gamma'(x)N_\gamma (\bar{\psi}_i)(x)| + (g_\gamma'(x))^2 |X_\gamma (x)\bar{\psi}_i'(g_\gamma(x))|+|X_\gamma (x)g_\gamma''(x) \bar{\psi}_i(g_\gamma(x))| \\
&\le cx^{\gamma-\alpha} (1-\log x),
\end{split}
\end{equation}
for $x\in (0, 1]$, where $c>0$ is a constant. Let $\delta \in (0,1)$ be as in the proof of Lemma~\ref{estim}. Then, it follows from~\eqref{uwr} that 
\[
|\partial_\gamma \L_\gamma \bar{\psi}_i(x)| \le \bar cx^{-\delta} \quad x\in (0, 1],
\]
where $\bar c>0$ is some constant. This implies that 
\[
\|\partial_\gamma \L_\gamma \bar \psi_i\|_{L^1(m)} \le \frac{\bar c}{1-\delta} \quad i\in\{1, 2\},
\]
and consequently (using that  $\L_\omega^\epsilon$ is a weak-contraction on $L^1(m)$)
\[
\left \| \sum_{r=1}^j (\L_{\sigma^{-j+r}\omega}^\epsilon)^{j-r} \int_{\beta(\sigma^{-j+r-1}\omega)}^{\beta(\sigma^{-j+r-1}\omega)+\epsilon \delta(\sigma^{-j+r-1}\omega)}\partial_\gamma \L_\gamma  (\L_{\sigma^{-j}\omega}^{r-1}\psi_i)\, d\gamma\right \|_{L^1(m)} \le n  \cdot \frac{\bar c \lvert \epsilon \rvert}{1-\delta}.
\]
We conclude that 
\[
\lim_{\epsilon \to 0}\left |\sum_{j=0}^n \delta(\sigma^{-(j+1)}\omega) \int_0^1\psi ((\L_{\sigma^{-j}\omega}^\epsilon )^j -\L_{\sigma^{-j}\omega}^j)[(X_{\beta(\sigma^{-(j+1)}\omega)} N_{\beta(\sigma^{-(j+1)}\omega)}(\L_{\sigma^{-l}\omega}^{l-1-j}(1)))']\, dm \right|=0.
\]

\paragraph{\bf{Bound on \eqref{eq:27dev}(II).}}
Observe that
\[
\begin{split}
&\bigg |\sum_{j=0}^n \delta(\sigma^{-(j+1)}\omega)\int_0^1 \psi \L_{\sigma^{-j}\omega}^j((X_{\beta(\sigma^{-(j+1)}\omega)} N_{\beta(\sigma^{-(j+1)}\omega)}(\L_{\sigma^{-l}\omega}^{l-1-j}(1)))'-\\
&-(X_{\beta(\sigma^{-(j+1)}\omega)} N_{\beta(\sigma^{-(j+1)}\omega)} (h(\sigma^{-(j+1)}\omega)))')\, dm\bigg | \\
&\le \sum_{j=0}^n \| \psi \|_{L^\infty(m)} \cdot \|\L_{\sigma^{-j}\omega}^j((X_{\beta(\sigma^{-(j+1)}\omega)} N_{\beta(\sigma^{-(j+1)}\omega)}(\L_{\sigma^{-l}\omega}^{l-1-j}(1)))'-\\
&\phantom{\le}
-(X_{\beta(\sigma^{-(j+1)}\omega)} N_{\beta(\sigma^{-(j+1)}\omega)} (h(\sigma^{-(j+1)}\omega)))')\|_{L^1(m)} \\
&\le \sum_{j=0}^n \| \psi\|_{L^\infty(m)} \cdot \| (X_{\beta(\sigma^{-(j+1)}\omega)} N_{\beta(\sigma^{-(j+1)}\omega)}(\L_{\sigma^{-l}\omega}^{l-1-j}(1)))'-\\
&\phantom{\le}-(X_{\beta(\sigma^{-(j+1)}\omega)} N_{\beta(\sigma^{-(j+1)}\omega)} (h(\sigma^{-(j+1)}\omega)))'\|_{L^1(m)}.
\end{split}
\]
Let us fix $0\leq j\leq n$. In what follows, we will show that the $j^{th}$ term above converges to $0$ as $l\to\infty$. To start, note that
\begin{equation}\label{eq:XNdiff}
\begin{split}
&\| (X_{\beta(\sigma^{-(j+1)}\omega)} N_{\beta(\sigma^{-(j+1)}\omega)}(\L_{\sigma^{-l}\omega}^{l-1-j}(1)))'-(X_{\beta(\sigma^{-(j+1)}\omega)} N_{\beta(\sigma^{-(j+1)}\omega)} (h(\sigma^{-(j+1)}\omega)))'\|_{L^1(m)} \\
&\le \|X_{\beta(\sigma^{-(j+1)}\omega)}' (N_{\beta(\sigma^{-(j+1)}\omega)}(\L_{\sigma^{-l}\omega}^{l-1-j}(1))-N_{\beta(\sigma^{-(j+1)}\omega)} (h(\sigma^{-(j+1)}\omega)))\|_{L^1(m)} \\
&\phantom{\le}+\|X_{\beta(\sigma^{-(j+1)}\omega)} (  (N_{\beta(\sigma^{-(j+1)}\omega)}(\L_{\sigma^{-l}\omega}^{l-1-j}(1)))'-(N_{\beta(\sigma^{-(j+1)}\omega)} (h(\sigma^{-(j+1)}\omega)))')\|_{L^1(m)}.
\end{split}
\end{equation}
Let $c>0$ be such that $\|X_\gamma'\|_\infty \le c$ for each $\gamma$. That such a bound exists follows from a direct calculation, e.g. using \eqref{eq:trivialboundong}. Then,  using that $N_\gamma$ is a contraction on $L^1(m)$ and~\eqref{dec}, it follows that the first term of
\eqref{eq:XNdiff} is bounded as follows,
\begin{equation}\label{ccalpha}
\begin{split}
&\|X_{\beta(\sigma^{-(j+1)}\omega)}' (N_{\beta(\sigma^{-(j+1)}\omega)}(\L_{\sigma^{-l}\omega}^{l-1-j}(1))-N_{\beta(\sigma^{-(j+1)}\omega)} (h(\sigma^{-(j+1)}\omega)))\|_{L^1(m)} \\
&\le c\|N_{\beta(\sigma^{-(j+1)}\omega)}(\L_{\sigma^{-l}\omega}^{l-1-j}(1))-N_{\beta(\sigma^{-(j+1)}\omega)} (h(\sigma^{-(j+1)}\omega))\|_{L^1(m)} \\
&\le c \|\L_{\sigma^{-l}\omega}^{l-1-j}(1)-h(\sigma^{-(j+1)}\omega)\|_{L^1(m)} \\
&=c\|\L_{\sigma^{-l}\omega}^{l-1-j}(1)-\L_{\sigma^{-l}\omega}^{l-1-j}(h(\sigma^{-l}\omega))\|_{L^1(m)} \\
&\le 2cC_\alpha (l-1-j)^{-1/\alpha+1} \\
&\le 2cC_\alpha (l-1-n)^{-1/\alpha+1}.
\end{split}
\end{equation}
Observe that the last term can be arbitrarily small by choosing $l$ sufficiently large. Thus, it remains to deal with the last term in \eqref{eq:XNdiff}. Before doing so, we introduce some notation. 
For $\omega\in \Omega'$ and $p\in \N_0$, set
\begin{equation}\label{b_pt}
b_p^\omega:=(T_{\omega}^p)^{-1}(1/2),
\end{equation}
where $T_{\omega}^p :=T_{\beta(\sigma^{p-1}\omega)}\circ \ldots \circ T_{\beta(\omega)}$ and the pre-images are taken with respect to restrictions $T_{\omega}^p \rvert_{[0, 1/2]}$. Then, $b_0^\omega=1/2$ and $b_{p+1}^\omega<b_p^\omega$. For $x\in (b_{p+1}^\omega, b_p^\omega]$, we have $T_{\omega}^{p+1}(x)\in (\frac 1 2, 1]$ and $T_{\omega}^j(x) \le \frac 1 2$ for $j\le p$. Moreover, 
\[
T_{\beta(\omega)}^{-1}(b_{p}^{\sigma \omega})=b_{p+1}^\omega.
\]
We will make use of the following lemma, whose proof is given in Subsection~\ref{controlbpt}.
\begin{lemma}\label{bpt}
There exist $C_\alpha>0$ depending only on $\alpha$ and $C_{\underline{\alpha}}>0$ depending only on $\underline {\alpha}$ such that 
\begin{equation}\label{first}
b_p^\omega-b_{p+1}^\omega \le b_{p+1}^\omega \le C_\alpha p^{-1/\alpha}
\end{equation}
and 
\begin{equation}\label{second}
b_p^\omega \ge C_{\underline \alpha} p^{-1/\underline{\alpha}},
\end{equation}
for $\omega \in \Omega'$ and $p\in \mathbb N$.
\end{lemma}
Going back to the last term in \eqref{eq:XNdiff}, we have  
\begin{equation}\label{eq:XNLdiff}
\begin{split}
&\|X_{\beta(\sigma^{-(j+1)}\omega)} (  (N_{\beta(\sigma^{-(j+1)}\omega)}(\L_{\sigma^{-l}\omega}^{l-1-j}(1)))'-(N_{\beta(\sigma^{-(j+1)}\omega)} (h(\sigma^{-(j+1)}\omega)))')\|_{L^1(m)}  \\
&=\int_0^{b_{p}^{\sigma^{-j}\omega}}|X_{\beta(\sigma^{-(j+1)}\omega)} (  (N_{\beta(\sigma^{-(j+1)}\omega)}(\L_{\sigma^{-l}\omega}^{l-1-j}(1)))'-(N_{\beta(\sigma^{-(j+1)}\omega)} (h(\sigma^{-(j+1)}\omega)))')| \, dm \\
&\phantom{\le}+\int_{b_{p}^{\sigma^{-j}\omega}}^1 |X_{\beta(\sigma^{-(j+1)}\omega)} (  (N_{\beta(\sigma^{-(j+1)}\omega)}(\L_{\sigma^{-l}\omega}^{l-1-j}(1)))'-(N_{\beta(\sigma^{-(j+1)}\omega)} (h(\sigma^{-(j+1)}\omega)))')| \, dm\\
&=: \eqref{eq:XNLdiff}(I)+\eqref{eq:XNLdiff}(II).
\end{split}
\end{equation}
\paragraph{\textbf{Bound on \eqref{eq:XNLdiff}(I).}}
 In the sequel, $c>0$ will denote a generic constant independent of $\omega, j$ and $l$ that can change its value from one occurrence to the next. Let $\phi \in \C_*(a)\cap \C_2$ with $\int_0^1 \phi \, dm=1$.   By Proposition~\ref{prop:conecontraction}, Remark~\ref{n06} and~\eqref{eq:trivialboundong} we have that 
\[
\begin{split}
|(N_{\beta(\sigma^{-(j+1)}\omega)}(\phi))'(x)| &\le \frac{b_1}{x}N_{\beta(\sigma^{-(j+1)}\omega)}(\phi) (x)\\
&=\frac{b_1}{x}g_{\beta(\sigma^{-(j+1)}\omega)}'(x)\phi(g_{\beta(\sigma^{-(j+1)}\omega)}(x))\\
&\le \frac{ab_1}{x} (g_{\beta(\sigma^{-(j+1)}\omega)}(x))^{-\alpha} \\
&\le 3ab_1 x^{-\alpha-1},
\end{split}
\]
for $x\in (0, 1]$. Hence, since $\L_{\sigma^{-l}\omega}^{l-1-j}(1), h(\sigma^{-(j+1)}\omega)\in \C_*(a)\cap \C_2$ satisfy $\int_0^1 \L_{\sigma^{-l}\omega}^{l-1-j}(1)\, dm=\int_0^1 h(\sigma^{-(j+1)}\omega)\, dm=1$, 
we have 
\[
|(N_{\beta(\sigma^{-(j+1)}\omega)}(\L_{\sigma^{-l}\omega}^{l-1-j}(1)))'(x)|\le cx^{-\alpha-1} \quad \text{and} \quad |(N_{\beta(\sigma^{-(j+1)}\omega)} (h(\sigma^{-(j+1)}\omega)))'(x)| \le cx^{-\alpha-1},
\]
for $x\in (0, 1]$.

We now have (using also~\cite[(2.3)]{BT}) that 
\[
\begin{split}
&|X_{\beta(\sigma^{-(j+1)}\omega)} (x)(  (N_{\beta(\sigma^{-(j+1)}\omega)}(\L_{\sigma^{-l}\omega}^{l-1-j}(1)))'-(N_{\beta(\sigma^{-(j+1)}\omega)} (h(\sigma^{-(j+1)}\omega)))')(x)|\\
&\le cx^{\beta(\sigma^{-(j+1)}\omega)-\alpha}(1-\log x), 
\end{split}
\]
for $x\in (0, 1]$. Using the arguments as in the proof of Lemma~\ref{estim}, we conclude that there exists $\delta \in (0, \alpha] $ such that
\[
|X_{\beta(\sigma^{-(j+1)}\omega)} (x)(  (N_{\beta(\sigma^{-(j+1)}\omega)}(\L_{\sigma^{-l}\omega}^{l-1-j}(1)))'-(N_{\beta(\sigma^{-(j+1)}\omega)} (h(\sigma^{-(j+1)}\omega)))')(x)| \le cx^{-\delta},
\]
for $x\in (0, 1]$. This implies that 
\begin{equation}\label{eq:int0bpkj}
\int_0^{b_{p}^{\sigma^{-j}\omega}}|X_{\beta(\sigma^{-(j+1)}\omega)} (  (N_{\beta(\sigma^{-(j+1)}\omega)}(\L_{\sigma^{-l}\omega}^{l-1-j}(1)))'-(N_{\beta(\sigma^{-(j+1)}\omega)} (h(\sigma^{-(j+1)}\omega)))')| \, dm
\end{equation}
can be made arbitrarily small by taking $p$ sufficiently large, as $\lim\limits_{p\to \infty}b_p^{\sigma^{-j}\omega}=0$.

\paragraph{\textbf{Bound on \eqref{eq:XNLdiff}(II).}}
To analyze this term, we begin by observing that since 
\[
(N_\gamma \phi)'=g_\gamma''\phi \circ g_\gamma+(g_\gamma')^2\phi'\circ g_\gamma=\frac{g_\gamma''}{g_\gamma'}N_\gamma \phi +g_\gamma' N_\gamma \phi',
\]
we have that 
\[
\begin{split}
& \int_{b_{p}^{\sigma^{-j}\omega}}^1 |X_{\beta(\sigma^{-(j+1)}\omega)} (  (N_{\beta(\sigma^{-(j+1)}\omega)}(\L_{\sigma^{-l}\omega}^{l-1-j}(1)))'-(N_{\beta(\sigma^{-(j+1)}\omega)} (h(\sigma^{-(j+1)}\omega)))')| \, dm \\
& \le \int_{b_{p}^{\sigma^{-j}\omega}}^1 \left  |X_{\beta(\sigma^{-(j+1)}\omega)}\frac{g_{\beta(\sigma^{-(j+1)}\omega)}''}{g_{\beta(\sigma^{-(j+1)}\omega)}'} \right |
\cdot |N_{\beta(\sigma^{-(j+1)}\omega)}(\L_{\sigma^{-l}\omega}^{l-1-j}(1)-h(\sigma^{-(j+1)}\omega))|\, dm \\
&\phantom{\le}+\int_{b_{p}^{\sigma^{-j}\omega}}^1
|X_{\beta(\sigma^{-(j+1)}\omega)}g_{\beta(\sigma^{-(j+1)}\omega)}'| \cdot |N_{\beta(\sigma^{-(j+1)}\omega)} ((\L_{\sigma^{-l}\omega}^{l-1-j}(1))'-(h(\sigma^{-(j+1)}\omega))')| \, dm \\
&\le c\int_0^1 |N_{\beta(\sigma^{-(j+1)}\omega)}(\L_{\sigma^{-l}\omega}^{l-1-j}(1)-h(\sigma^{-(j+1)}\omega))|\, dm \\
&\phantom{\le}+c\int_{b_{p}^{\sigma^{-j}\omega}}^1 g_{\beta(\sigma^{-(j+1)}\omega)}'|N_{\beta(\sigma^{-(j+1)}\omega)} ((\L_{\sigma^{-l}\omega}^{l-1-j}(1))'-(h(\sigma^{-(j+1)}\omega))')| \, dm \\
&\le c\|\L_{\sigma^{-l}\omega}^{l-1-j}1-h(\sigma^{-(j+1)}\omega)\|_{L^1(m)}+c\int_{b_{p+1}^{\sigma^{-(j+1)}\omega}}^{1/2}|(\L_{\sigma^{-l}\omega}^{l-1-j} (1))'-(h(\sigma^{-(j+1)}\omega))'| \, dm,
\end{split}
\]
where we have used the change of variables and the fact that $N_\beta$ is a weak contraction on $L^1(m)$.  Note that~\eqref{dec} implies that 
\[
\|\L_{\sigma^{-l}\omega}^{l-1-j}1-h(\sigma^{-(j+1)}\omega)\|_{L^1(m)} \le 2C_\alpha (l-1-n)^{-1/\alpha+1},
\]
which can be made arbitrarily small by choosing $l$ sufficiently large. 
Next,
\[
\begin{split}
&\int_{b_{p+1}^{\sigma^{-(j+1)}\omega}}^{1/2}|(\L_{\sigma^{-l}\omega}^{l-1-j} (1))'-(h(\sigma^{-(j+1)}\omega))'| \, dm\\
&=\int_{b_{p+1}^{\sigma^{-(j+1)}\omega}}^{1/2}|( \L_{\sigma^{-l}\omega}^{l-1-j}(1-h(\sigma^{-l}\omega)))'| \, dm \\
&=
\int_{b_{p+1}^{\sigma^{-(j+1)}\omega}}^{1/2}|(\L^{\ell}_{\sigma^{-(j+\ell+1)}\omega}(\L_{\sigma^{-l}\omega}^{l-1-j-\ell}1-h(\sigma^{-(j+\ell+1)}\omega)))'| \, dm,
\end{split}
\]
where $\ell=\lfloor \frac{l-1-j}{2}\rfloor$.
For subsequent arguments, we note that since $j\le n$ is fixed, we have $\ell \to \infty$ as $l\to \infty$.
By writing $\varphi:=\L_{\sigma^{-l}\omega}^{l-1-j-\ell}1$ we have that 
\[
\begin{split}
&( \L_{\sigma^{-(j+\ell+1)}\omega}^{\ell}(\varphi-h(\sigma^{-(j+\ell+1)}\omega)))'  \\
&=
\L_{\sigma^{-(j+\ell+1)}\omega}^{\ell} \left (\frac{1}{(T_{\sigma^{-(j+\ell+1)}\omega}^{\ell})'}(\varphi-h(\sigma^{-(j+\ell+1)}\omega))'+\frac{(T_{\sigma^{-(j+\ell+1)}\omega}^{\ell})''}{((T_{\sigma^{-(j+\ell+1)}\omega}^{\ell})')^2}(\varphi-h(\sigma^{-(j+\ell+1)}\omega)) \right ).
\end{split}
\]
Then,
\begin{equation}\label{eq:almostlast2}
\begin{split}
&\int_{b_{p+1}^{\sigma^{-(j+1)}\omega}}^{1/2}|(\L_{\sigma^{-l}\omega}^{l-1-j} (1))'-(h(\sigma^{-(j+1)}\omega))'| \, dm \\
&\leq
\int_{b_{p+1}^{\sigma^{-(j+1)}\omega}}^{1/2} \left |\L_{\sigma^{-(j+\ell+1)}\omega}^{\ell} \left (\frac{1}{(T_{\sigma^{-(j+\ell+1)}\omega}^{\ell})'}(\varphi-h(\sigma^{-(j+\ell+1)}\omega))' \right ) \right | \, dm \\
& \phantom{\leq}+
\int_{b_{p+1}^{\sigma^{-(j+1)}\omega}}^{1/2} \left |\L_{\sigma^{-(j+\ell+1)}\omega}^{\ell} \left (\frac{(T_{\sigma^{-(j+\ell+1)}\omega}^{\ell})''}{((T_{\sigma^{-(j+\ell+1)}\omega}^{\ell})')^2}(\varphi-h(\sigma^{-(j+\ell+1)}\omega)) \right )\right |\, dm \\
&=: \eqref{eq:almostlast2}(I)+\eqref{eq:almostlast2}(II). 
\end{split}
\end{equation}

\paragraph{\textbf{Bound on \eqref{eq:almostlast2}(I).}} 
We will need the following lemma. Its proof will be deferred until Subsection~\ref{Lem19}.
\begin{lemma}\label{newlem1}
Let $p\ge 0$. Then, 
there exists a constant $C_{p, \alpha, \underline{\alpha}} > 0$ such that for $\omega \in \Omega'$, $\ell \ge 1$, and for $0 \le k \le p$, 
$$
\frac{1}{  (T_\omega^\ell )' (x)  } \le C_{p, \alpha, \underline{\alpha}} b^\omega_{ \ell + k  }  \le C_{p, \alpha, \underline{\alpha}}(\ell+k)^{-1/\alpha},
$$
for all but finitely many $x$ with $T_\omega^\ell(x) \in (b_{k}^{\sigma^\ell \omega} ,  b_{k-1}^{\sigma^\ell \omega } )$,
where we set $b_{-1}^{\omega} := 1$.
\end{lemma}
Using Lemma~\ref{newlem1} we have 
\[
\begin{split}
&\int_{b_{p+1}^{\sigma^{-(j+1)}\omega}}^{1/2} \left |\L_{\sigma^{-(j+\ell+1)}\omega}^{\ell} \left (\frac{1}{(T_{\sigma^{-(j+\ell+1)}\omega}^{\ell})'}(\varphi-h(\sigma^{-(j+\ell+1)}\omega))' \right ) \right | \, dm \\
&\le \int_{b_{p+1}^{\sigma^{-(j+1)}\omega}}^{1/2}
\L_{\sigma^{-(j+\ell+1)}\omega}^{\ell}\left (\frac{|(\varphi-h(\sigma^{-(j+\ell+1)}\omega))'|}{(T_{\sigma^{-(j+\ell+1)}\omega}^{\ell})'} \right ) \, dm \\
&=\int_{(T_{\sigma^{-(j+\ell+1)}\omega}^\ell)^{-1}([b_{p+1}^{\sigma^{-(j+1)}\omega}, \frac 1 2])}\frac{|(\varphi-h(\sigma^{-(j+\ell+1)}\omega))'|}{(T_{\sigma^{-(j+\ell+1)}\omega}^{\ell})'}\, dm \\
&=\sum_{q=1}^{p+1}\int_{(T_{\sigma^{-(j+\ell+1)}\omega}^\ell)^{-1}([b_{q}^{\sigma^{-(j+1)}\omega},  b_{q-1}^{\sigma^{-(j+1)}\omega}])}\frac{|(\varphi-h(\sigma^{-(j+\ell+1)}\omega))'|}{(T_{\sigma^{-(j+\ell+1)}\omega}^{\ell})'}\, dm \\
&\le C_{p, \alpha, \underline{\alpha}}\sum_{q=1}^{p+1}b_{\ell+q}^{\sigma^{-(j+\ell+1)}\omega}\int_{(T_{\sigma^{-(j+\ell+1)}\omega}^\ell)^{-1}([b_{q}^{\sigma^{-(j+1)}\omega},  b_{q-1}^{\sigma^{-(j+1)}\omega}])}|(\varphi-h(\sigma^{-(j+\ell+1)}\omega))'|\, dm \\
&\le C_{p, \alpha, \underline{\alpha}}\sum_{q=1}^{p+1}b_{\ell+q}^{\sigma^{-(j+\ell+1)}\omega} \int_{b_{\ell+q}^{\sigma^{-(j+\ell+1)}\omega}}^1|(\varphi-h(\sigma^{-(j+\ell+1)}\omega))'|\, dm \\
&\le C_{p, \alpha, \underline{\alpha} }\sum_{q=1}^{p+1}b_{\ell+q}^{\sigma^{-(j+\ell+1)}\omega} \int_{b_{\ell+q}^{\sigma^{-(j+\ell+1)}\omega}}^1 x^{-1-\alpha}\, dx \\
&\le C_{p, \alpha, \underline{\alpha}}\sum_{q=1}^{p+1}(b_{\ell+q}^{\sigma^{-(j+\ell+1)}\omega})^{1-\alpha} \\
&\le C_{p, \alpha, \underline{\alpha}}(b_{\ell}^{\sigma^{-(j+\ell+1)}\omega})^{1-\alpha},
\end{split}
\]
which goes to $0$ as $\ell \to \infty$.

\paragraph{\textbf{Bound on \eqref{eq:almostlast2}(II).}}
We will use the following lemma whose proof is included in Subsection~\ref{Lem20}.
\begin{lemma}\label{lem:dist-2} There exists a constant $C_{p, \alpha, \underline{\alpha}} > 0$ such that for $\omega \in \Omega'$, $\ell \ge 1$, and for 
$1 \le k \le p$, 
\begin{align}
	\frac{(T_\omega^\ell)''(x)}{ (  ( T^\ell_\omega  )'(x) )^2 } \le C_{p, \alpha, \underline{\alpha}}  \ell^{ 1 - \underline{\alpha} / \alpha }
\end{align}
holds
whenever $T_\omega^\ell(x) \in (b_{k}^{\sigma^\ell \omega} ,  b_{k-1}^{\sigma^\ell \omega} )$, except possibly for a finite 
number of points $x$.
\end{lemma}
By Lemma~\ref{lem:dist-2} and~\eqref{dec} we have
\[
\begin{split}
&\int_{b_{p+1}^{\sigma^{-(j+1)}\omega}}^{1/2} \left |\L_{\sigma^{-(j+\ell+1)}\omega}^{\ell} \left (\frac{(T_{\sigma^{-(j+\ell+1)}\omega}^{\ell})''}{((T_{\sigma^{-(j+\ell+1)}\omega}^{\ell})')^2}(\varphi-h(\sigma^{-(j+\ell+1)}\omega)) \right )\right |\, dm \\
&\le \int_{b_{p+1}^{\sigma^{-(j+1)}\omega}}^{1/2}\sum_{y\in (T_{\sigma^{-(j+\ell+1)}\omega}^\ell)^{-1}(x)}\frac{\frac{|(T_{\sigma^{-(j+\ell+1)}\omega}^{\ell})''(y)|}{((T_{\sigma^{-(j+\ell+1)}\omega}^{\ell})')^2(y)} \cdot |(\varphi-h(\sigma^{-(j+\ell+1)}\omega))(y)| }{(T_{\sigma^{-(j+\ell+1)}\omega}^\ell )'(y)}\, dm(x) \\
&=\sum_{q=1}^{p+1}\int_{[b_q^{\sigma^{-(j+1)}\omega}, b_{q-1}^{\sigma^{-(j+1)}\omega}]}
\sum_{y\in (T_{\sigma^{-(j+\ell+1)}\omega}^\ell)^{-1}(x)}\frac{\frac{|(T_{\sigma^{-(j+\ell+1)}\omega}^{\ell})''(y)|}{((T_{\sigma^{-(j+\ell+1)}\omega}^{\ell})')^2(y)} \cdot |(\varphi-h(\sigma^{-(j+\ell+1)}\omega))(y)| }{(T_{\sigma^{-(j+\ell+1)}\omega}^\ell )'(y)}\, dm(x) \\
&\le C_{p, \alpha, \underline{\alpha}}\ell^{1-\underline{\alpha}/\alpha} \int_0^1 
\L_{\sigma^{-(j+\ell+1)}\omega}^{\ell} (|  \varphi - h(\sigma^{-(j+\ell+1)}\omega) |) \, dm  \\
&=C_{p, \alpha, \underline{\alpha}}\ell^{1-\underline{\alpha}/\alpha}\|\varphi-h(\sigma^{-(j+\ell+1)}\omega)\|_{L^1(m)} \\
&=C_{p, \alpha, \underline{\alpha}}\ell^{1-\underline{\alpha}/\alpha}\|\L_{\sigma^{-l}\omega}^{l-1-j-\ell}(1-h(\sigma^{-l}\omega))\|_{L^1(m)} \\
&\le C_{p, \alpha, \underline{\alpha}}\ell^{1-\underline{\alpha}/\alpha} (l-1-j-\ell)^{1-1/\alpha} \\
&\le C_{p, \alpha, \underline{\alpha}}\ell^{2-(1+\underline{\alpha})/\alpha},
\end{split}
\]
which goes to $0$ when $\ell \to \infty$ as $\alpha<\frac 1 2$.
\end{proof}

\begin{remark}
A careful inspection of the proof of Theorem~\ref{thm:LR} yields the existence of $C_{\alpha, \underline{\alpha}}>0$ such that, for  $\psi \in L^\infty(m)$, $\omega \in \Omega'$ and $\varepsilon \neq 0$ sufficiently close to $0$,
\begin{equation}\label{refined}
\left |\frac{1}{\epsilon} \int_0^1\psi\left(h_\epsilon(\omega)-h(\omega)\right)dm-L\right |\le C_{\alpha, \underline{\alpha}} 
|\epsilon|^{1-2\alpha}\|\psi\|_{L^\infty(m)},
\end{equation}
where $L=L(\omega)$ is given by~\eqref{eq:L}. Indeed, 
as in~\eqref{eq:diffder} (taking $\xi=1$) we have that there exists $l_0=l_0(\epsilon)$ such that for $l\ge l_0$,
\[
\begin{split}
\left |\frac{1}{\epsilon} \int_0^1\psi\left(h_\epsilon(\omega)-h(\omega)\right)dm-L\right | & \le \left |\frac{1}{\varepsilon} \int_0^1  \psi ((\L_{\sigma^{-l}\omega}^\epsilon)^l 1- \L_{\sigma^{-l}\omega}^l1)\, dm-L\right |\\
&\phantom{\le}+C_\alpha |\varepsilon|\|\psi\|_{L^\infty(m)}.
\end{split}
\]
Recall that the term \[\frac{1}{\varepsilon} \int_0^1  \psi ((\L_{\sigma^{-l}\omega}^\epsilon)^l 1- \L_{\sigma^{-l}\omega}^l1)\, dm-L\]
is the sum of the terms~\eqref{E1} and~\eqref{E2}, where the absolute value of~\eqref{E2} is bounded by $C_\alpha|\epsilon| \|\psi\|_{L^\infty(m)}$.
Similarly to~\eqref{eq:27dev} and by choosing $n=\lfloor |\epsilon|^{-\alpha}\rfloor$, we have that for $l>n$,
\[
\begin{split}
&\left | \sum_{j=0}^{l-1}\delta(\sigma^{-(j+1)}\omega)\int_0^1 \psi (\L_{\sigma^{-j}\omega}^\epsilon )^j\partial_\gamma \L_{\beta(\sigma^{-(j+1)}\omega)} (\L_{\sigma^{-l}\omega}^{l-1-j}(1))\, dm  -L \right | \\
&\le \left |\sum_{j=0}^n \delta (\sigma^{-(j+1)}\omega)\int_0^1\psi ((\L_{\sigma^{-j}\omega}^\epsilon )^j -\L_{\sigma^{-j}\omega}^j)[(X_{\beta(\sigma^{-(j+1)}\omega)} N_{\beta(\sigma^{-(j+1)}\omega)}(\L_{\sigma^{-l}\omega}^{l-1-j}(1)))']\, dm \right| \\
&\phantom{\le}+
\bigg |\sum_{j=0}^n \delta(\sigma^{-(j+1)}\omega)\int_0^1 \psi \L_{\sigma^{-j}\omega}^j((X_{\beta(\sigma^{-(j+1)}\omega)} N_{\beta(\sigma^{-(j+1)}\omega)}(\L_{\sigma^{-l}\omega}^{l-1-j}(1)))'-\\
&\phantom{\le}
-(X_{\beta(\sigma^{-(j+1)}\omega)} N_{\beta(\sigma^{-(j+1)}\omega)} (h(\sigma^{-(j+1)}\omega)))')\, dm\bigg |+
C_\alpha |\epsilon|^{1-2\alpha}\|\psi\|_{L^\infty(m)},
\end{split}
\]
since $\sum\limits_{j=n+1}^\infty j^{1-1/\alpha}=O(n^{2-1/\alpha})=O(|\epsilon|^{1-2\alpha})$.
As in the proof of Theorem~\ref{thm:LR},
\[
\begin{split}
&\left |\sum_{j=0}^n \delta (\sigma^{-(j+1)}\omega)\int_0^1\psi ((\L_{\sigma^{-j}\omega}^\epsilon )^j -\L_{\sigma^{-j}\omega}^j)[(X_{\beta(\sigma^{-(j+1)}\omega)} N_{\beta(\sigma^{-(j+1)}\omega)}(\L_{\sigma^{-l}\omega}^{l-1-j}(1)))']\, dm \right| \\
&\le C_{\alpha, \underline{\alpha}}n^2|\epsilon|\|\psi\|_{L^\infty(m)}=C_{\alpha, \underline{\alpha}}|\epsilon|^{1-2\alpha}\|\psi\|_{L^\infty(m)}.
\end{split}
\]
Moreover, the arguments in the proof of Theorem~\ref{thm:LR} show that for $0\le j \le n$ and $\omega \in \Omega'$, 
\[
\begin{split}
\| (X_{\beta(\sigma^{-(j+1)}\omega)} N_{\beta(\sigma^{-(j+1)}\omega)}(\L_{\sigma^{-l}\omega}^{l-1-j}(1)))'
-(X_{\beta(\sigma^{-(j+1)}\omega)} N_{\beta(\sigma^{-(j+1)}\omega)} (h(\sigma^{-(j+1)}\omega)))'\|_{L^1(m)} \le |\epsilon|.
\end{split}
\]
For example, the first term in~\eqref{eq:XNdiff} can be made $\le |\epsilon|/2$ by choosing $l$ large enough (see~\eqref{ccalpha}), and the other term can be handled in the same way by choosing (independently on $\omega$)  $l$ or $p$ large enough. Hence, 
\[
\begin{split}
&\bigg |\sum_{j=0}^n \delta(\sigma^{-(j+1)}\omega)\int_0^1 \psi \L_{\sigma^{-j}\omega}^j((X_{\beta(\sigma^{-(j+1)}\omega)} N_{\beta(\sigma^{-(j+1)}\omega)}(\L_{\sigma^{-l}\omega}^{l-1-j}(1)))'-\\
&\phantom{\le}
-(X_{\beta(\sigma^{-(j+1)}\omega)} N_{\beta(\sigma^{-(j+1)}\omega)} (h(\sigma^{-(j+1)}\omega)))')\, dm\bigg | \le n|\epsilon| \|\psi\|_{L^\infty(m)}.
\end{split}
\] 
Putting all the estimates together yields~\eqref{refined}.
\end{remark}
Finally, we observe that~\eqref{refined} yields the following annealed linear response result.
\begin{cor}\label{cor:ALR}
In the context of Theorem~\ref{thm:LR}, the following holds.
\[
\lim_{\epsilon \to 0}\frac{1}{\epsilon}\int_\Omega\int_0^1\Psi(\omega, \cdot)(h_\epsilon(\omega)-h(\omega))\, dm\, d\mathbb P(\omega)=\int_\Omega L(\omega)\, d\mathbb P(\omega),
\]
for $\Psi\in L^\infty(\mathbb P\times m)$, where  $L(\omega)$ is as in~\eqref{eq:L} with $\psi=\Psi(\omega, \cdot)$.
\end{cor}

\subsection{Linear response for sequential dynamics and non-uniqueness of sequential a.c.i.m}\label{ssec:seq}
The arguments developed in the previous sections apply also to sequential dynamics formed by nonautonomous compositions of the form
\[
T_{\beta_{k+n-1}}\circ \ldots \circ T_{\beta_{k+1}}\circ T_{\beta_k},
\]
where $k\in \Z$, $n\in \N$ and $(\beta_k)_{k\in \Z}\subset (0, 1)$.

\begin{proposition}\label{ABX}
    Let $(\beta_k)_{k\in \Z}\subset (0, 1)$ be such that $\alpha:=\sup_{k\in \Z}\beta_k<1$. Then, there exists a  sequence $(h_k)_{k\in \Z}\subset \C_\ast(a)\cap \C_2$ such that 
    \begin{equation}\label{sim}
    \mathcal L_k h_k=h_{k+1} \quad \text{and} \quad \int_0^1h_k\, dm=1 \quad \text{for $k\in \Z$,}
    \end{equation}
    where $\L_k$ denotes the transfer operator associated with $T_{\beta_k}$. Moreover, $(h_k)_{k\in \Z}$ is the unique sequence in $\C_\ast(a)$ satisfying~\eqref{sim}.
\end{proposition}

\begin{proof}
The proof can be established by arguing as in the proof of claim~\ref{p1-1}  of Proposition~\ref{p1}, and by replacing $\psi_n^\omega$ with $\psi_n^k=\L_{k-n}^n 1$, where 
\[
\L_n^m=\L_{m+n-1}\circ \ldots \circ \L_{n+1}\circ \L_n, \quad \text{for $n\in \Z$ and $m\in \N$.}
\]
\end{proof}

The sequence $(h_k)_{k\in \Z}$ given by Proposition~\ref{ABX} can be regarded as a sequential a.c.i.m associated with the sequence $(T_{\beta_k})_{k\in \Z}$. 
In fact, the argument in the previous proof shows that $(h_k dm)_{k\in \Z}$ is the unique SRB state in the sense of Ruelle~\cite[Section 4]{Ruelle-diffSRB}. Namely, $\mu_k:=h_k \,dm$ satisfies 
$$\mu_k=\lim_{n\to \infty} T_{\beta_{k-n}}^* T_{\beta_{k-n+1}}^*\dots T_{\beta_{k-1}}^* m.$$
However, the analogy one can draw with the deterministic case is subtle, as  there are infinitely many sequential a.c.i.m's with densities in $L^1(m)$. This is a consequence of the following general result: 
\begin{proposition}[Non-uniqueness of sequential a.c.i.m]\label{L:non-uniq}
Assume $(\mathcal L_n)_{n\in \Z}$ is a sequence of transfer operators in $L^1(m)$, associated with surjective, finite-branched and nonsingular maps $T_n\colon [0, 1] \to [0, 1]$ with derivative bounded above and away from zero. 
Suppose $(h_n)_{n\in \Z}$ is a sequence of densities in $L^1(m)$ such that $\mathcal L_n h_n=h_{n+1}$ for $n\in \Z$. Furthermore, assume that there exists $c>0$ such that $h_0\geq c>0$.
Then, there exist uncountably many sequences of densities $(\tilde h_n)_{n\in \Z}\subset L^1(m)$ verifying $\mathcal L_n\tilde h_n=\tilde h_{n+1}$ for $n\in \Z$.
\end{proposition}
\begin{proof}
Let $0\neq \psi \in L^\infty$ be such that $\|\psi\|_{L^\infty}<c$ and $\int_0^1 \psi dm = 0$. Observe that there are uncountably many possible choices.
Let $\tilde h_0= h_0 + \psi$. Then $\tilde h_0 \neq h_0$ is a density.
For $j>0$, let $\tilde h_j= \mathcal L_j \tilde h_{j-1}$. 
For $j<0$, let $\tilde h_j$ be a density such that $\mathcal L_j \tilde h_j=\tilde h_{j+1}$. This choice is possible, as $\mathcal L_j$ is onto in $L^1$. 
In fact, if $h\in L^1$, then $\eta(x):=\tfrac{1}{n(x)}|T_{j}'(x)| h( T_{j}(x))$, where $n(x)$ is the number of preimages of $x$ under $T_j$ satisfies $\eta \in L^1$ and
$\mathcal L_j(\eta) = h$. Furthermore, if $h$ is a density, so is $\eta$.
In this way, $(\tilde h_n)_{n\in \Z}\subset L^1(m)$  is a sequence of densities satisfying $\mathcal L_n \tilde h_n= \tilde h_{n+1}$ for $n\in \Z$. 
\end{proof}

Our intermittent sequential dynamics satisfies the assumptions in Proposition \ref{L:non-uniq}. Indeed, it is sufficient to observe that the arguments in the proof of Lemma~\ref{TL} apply here as well: one can then argue as in the proof of claim~\ref{p1-2} of Proposition~\ref{p1}, and conclude that there exists $c>0$ such that $h_k\ge c$ for $k\in \Z$.

We emphasize that the above construction is applicable to the case where $(\beta_k)_{k\in \Z}$ is a constant sequence. Therefore, even if our sequential dynamics is generated by compositions of a single map $T_\beta$ with itself, one still has infinitely many sequential a.c.i.m's. However, only one of them lies in $\C_\ast(a)$ and is the constant sequence $(h)_{k\in \Z}$, where $h$ is the density of the unique  (deterministic) a.c.i.m for $T_\beta$. 

On the other hand, we observe that each sequential a.c.i.m $(\bar h_k)_{k\in \Z}\subset L^1(m)$ satisfies the following:
\begin{equation}\label{asyequiv}
\lim_{k\to \infty}\|h_k-\bar h_k\|_{L^1(m)}=0,
\end{equation}
where $(h_k)_{k\in\Z}\subset \mathcal C_*(a)$ is the unique SRB state from Proposition~\ref{ABX}.
Indeed, choose $\epsilon>0$ arbitrary and let $\varphi \in C^1[0, 1]$ be such that $\|\varphi-\bar h_0\|_{L^1(m)}\le \epsilon$. We can assume that $\int_0^1\varphi\, dm=1$.  Then, 
\[
\begin{split}
\|h_k-\bar h_k\|_{L^1(m)}=\|\L_0^k(h_0- \bar h_0)\| & \le \|\L_0^k(h_0-\varphi)\|_{L^1(m)}+\|\L_0^k(\varphi-\bar h_0)\|_{L^1(m)}\\
&\le \|\L_0^k(h_0-\varphi)\|_{L^1(m)}+\|\varphi-\bar h_0\|_{L^1(m)}\\
&\le \|\L_0^k(h_0-\varphi)\|_{L^1(m)}+\epsilon.
\end{split}
\]
Next, as $\varphi$ is H\"{o}lder and $h_0\in \C_\ast(a)$, it follows from~\cite[Theorem 1.1]{KL} that $\lim\limits_{k\to \infty}\|\L_0^k(h_0-\varphi)\|_{L^1(m)}=0$. Since $\epsilon>0$ was arbitrary, we conclude that~\eqref{asyequiv} holds.

Since under the assumptions of Proposition~\ref{ABX} we have the uniqueness of sequential a.c.i.m in the cone $\C_{\ast}(a)$, which coincides with Ruelle's SRB state, it makes sense to consider the linear response restricted to the class of sequential a.c.i.m's which belong to this cone.
The following result can be established by repeating the arguments in the proof of Theorem~\ref{thm:LR}.
\begin{theorem}%[Linear response for SRB states]
\label{thm:seqLR}
Let $0<\underline{\alpha}<\alpha<\frac12$ and $\epsilon_0 >0$. For each $k\in \mathbb Z$, assume $\beta_k, \delta_k$ satisfy the following: $0\leq  \delta_k<1$,
                $\underline{\alpha}\leq\beta_k-\epsilon_0$ and $\beta_k+\epsilon_0\le \alpha$.
                Then, for every $k\in \mathbb Z$ and $\psi\in L^\infty(m)$,
                \[
                \lim_{\epsilon\to0}\frac{1}{\epsilon} \int_0^1\psi\left(h^\epsilon_{k}-h_k\right)dm=
                - \sum_{i=0}^\infty \delta_{k-i-1} \int_0^1\psi \, \L_{k-i}^i \left(X_{\beta_{k-i-1}} N_{\beta_{k-i-1}} (h_{k-i-1}) \right)'\, dm, 
                \]
                where $(h^\epsilon_{k})_{k\in\mathbb Z}\subset \C_\ast(a)$ is the sequential a.c.i.m. 
                associated with $T^\epsilon=(T_{\beta_k+\epsilon \delta_k})_{k\in\mathbb Z}$ for $\epsilon \in (-\epsilon_0, \epsilon_0)$, and $h_k:=h_k^0$, $k\in \Z$.
                \end{theorem}

\subsection{A technical lemma}
We use the same notation as in Subsection~\ref{racim}.
\begin{lemma}\label{TL}
There exists $N\in \mathbb N$ and $c>0$ such that 
\[
\mathcal L_\omega^N \varphi\ge c, \quad \text{for  $\omega \in \Omega'$ and $\varphi\in \mathcal C_*(a)$ with $\int_0^1 \varphi\, dm=1$.}
\]
\end{lemma}

\begin{proof}
We choose $\delta \in (0, 1)$ such that $a\delta^{1-\alpha}=\frac 1 2$. Note that 
\[
\int_0^\delta \varphi\, dm\le a\delta^{1-\alpha}\int_0^1 \varphi\, dm=a\delta^{1-\alpha}=\frac 1 2,
\]
and consequently
\[
\int_\delta^1 \varphi\, dm=1-\int_0^\delta \varphi\, dm\ge \frac 1 2.
\]
Since $\varphi$ is decreasing, 
\[
\varphi(x)\ge \varphi(\delta)\ge \frac{\int_\delta^1 \varphi\, dm}{1-\delta} \ge \frac{1}{2(1-\delta)}, \quad \forall x\in (0, \delta].
\]
Let $b_p^\omega$ be as in~\eqref{b_pt}. Since $b_{p+1}^\omega \le C_\alpha p^{-1/\alpha}$ (see~\eqref{first}), there exists $N\in \mathbb N$ such that 
\[
b_{N-1}^\omega <\delta, \quad \text{for  $\omega \in \Omega'$.}
\]
Note that $(0, b_{N-1}^\omega]$ is the far-left interval of monotonicity for $T_\omega^N$ that is mapped bijectively onto $(0, 1]$. For $x\in (0, 1]$, by $y_N(\omega, x)$ we denote the unique  element from 
$(0, b_{N-1}^\omega]$ such that $T_\omega^N(y_N(\omega, x))=x$. Then, since $\varphi\ge 0$ we have 
\[
(\mathcal L_\omega^N \varphi)(x)\ge \frac{\varphi(y_N(\omega, x))}{ (T_\omega^N)'(y_N(\omega, x))}\ge \frac{1}{2(1-\delta) (T_\omega^N)'(y_N(\omega, x))}\ge \frac{1}{2(1-\delta)3^N}=:c>0.
\]
\end{proof}

\subsection{Proof of Lemma~\ref{bpt}}\label{controlbpt}
The first assertion of the lemma follows from~\cite[Proposition 3.13]{KL}. We now establish the second conclusion. Using that $\gamma \mapsto T_\gamma(x)$ is decreasing for $x\in (0, \frac 1 2]$ and that $x\mapsto T_\gamma(x)$ is increasing on $[0, \frac 1 2]$, one can show that $b_p^\omega\ge b_p$ where 
$b_p$ denotes the preimage of $\frac 1 2$ under the first branch of the composition of $T_{\underline{\alpha}}$ with itself $p$-times. The desired conclusion now follows from~\cite[Lemma 2.2]{L}.
\hfill $\qedsymbol$

\subsection{Proof of Lemma~\ref{newlem1}}\label{Lem19}
We will need the following distortion bound that was established in~\cite[Corollary 3.3.]{bose2021random}. We stress that although in~\cite{bose2021random} the authors consider random i.i.d compositions of LSV maps, the proof of~\cite[Corollary 3.3.]{bose2021random} applies to our setting as well. 
\begin{lemma}\label{lem:distortion}  There exists $C_{\alpha, \underline{\alpha}}> 0$  such that
	for any $\omega \in \Omega'$, $n\in \N$ and an interval $[x,y] \subset [0,1)$ that is bijectively mapped by $T_\omega^n$ onto $T_\omega^n ([x,y]) \subset [\tfrac12, 1)$, we have
	\begin{align}\label{eq:distort}
		| \log (T_\omega^n)'(x) - \log (T_\omega^n)'(y) | \le C_{\alpha, \underline{\alpha}} | T_\omega^n (x) - T_\omega^n (y) |.
	\end{align}
\end{lemma}
In addition, as in~\cite{leppanen2017functional}  for each $\omega \in \Omega'$, $n \in \N$, there exists a (mod $m$) partition 
$$
\{  I_{\omega}^{(n)} (\theta): \ \theta \in \{1, \ldots, 2^n\} \}
$$
of $[0,1]$ into open subintervals $I_{\omega}^{(n)} (\theta)$ such that $T_\omega^n$ maps each $ I_{\omega}^{(n)}$ diffeomorphically
onto $(0,1)$. We let $I_{\omega}^{(n)}(1) = (0, b^\omega_{n-1} )$ be the interval whose left end-point is zero. By \cite[Eq. (4)]{leppanen2017functional}, we have
\begin{align}\label{eq:compare_intervals}
|I_\omega^{(n)}(\theta)| \le |I_{\omega}^{(n)}(1)| =  b_{n-1}^\omega \quad \text{for $ \theta=1, \ldots, 2^n$,}
\end{align}
where $|I_\omega^{(n)}(\theta)|$ denotes the length of $I_\omega^{(n)}(\theta)$.
We are now in a position to prove Lemma~\ref{newlem1}. Suppose that $x \in I_{\omega}^{(\ell + k + 1)}(\theta)$, $\theta \in \{1,\ldots, 2^{\ell + k + 1} \}$.
If $T_\omega^\ell(x) \in (b_{k}^{\sigma^\ell \omega} ,  b_{k-1}^{\sigma^\ell \omega} )$, then $T_\omega^{\ell}$ maps $I_{\omega}^{(\ell + k + 1)}(\theta)$ bijectively onto $( b_k^{\sigma^\ell \omega}, b_{k-1}^{\sigma^\ell \omega} )$, and $T_\omega^{\ell + k}$ maps $I_{\omega}^{(\ell + k + 1)}(\theta)$ bijectively onto
$(1/2,1)$. By \eqref{eq:distort},
\begin{align*}
	\frac12 &= \int_{  I_\omega^{( k + \ell + 1 )} ( \theta ) }   (T_\omega^{\ell + k})' ( \xi) \, d \xi = 
	 (T_\omega^{\ell + k})' ( x ) \int_{  I_\omega^{( k + \ell + 1 )} ( \theta ) }   \frac{  (T_\omega^{\ell + k})' ( \xi) }{  (T_\omega^{\ell + k})' ( x ) }   \, d \xi \\
	 &\le C_{ \alpha, \underline{\alpha} }   (T_\omega^{\ell + k})' ( x ) |   I_\omega^{( k + \ell + 1 )} (\theta)  |.
\end{align*}
Using \eqref{eq:compare_intervals}, we obtain
$$ 
\frac{1}{ ( T_{\sigma^\ell \omega}^k )' ( T^\ell_\omega (x) ) }  \cdot \frac{1}{  (T_\omega^\ell)'(x)  }  =
\frac{1}{(T_\omega^{\ell + k})' ( x )} \le    C_{ \alpha, \underline{\alpha} }   |   I_\omega^{( k + \ell + 1 )} (\theta)  | \le 
 C_{ \alpha, \underline{\alpha} }  b_{ k + \ell  }^\omega.
$$
Since $k \le p$, we conclude that
$$
\frac{1}{  (T_\omega^\ell)'(x)  } \le  C_{p, \alpha, \underline{\alpha} } b^\omega_{ k + \ell }.
$$
Now the second inequality in the statement of Lemma~\ref{newlem1} follows readily from~\eqref{first}.
\hfill $\qedsymbol$

\subsection{Proof of Lemma~\ref{lem:dist-2}}\label{Lem20}
We want to estimate 
\[
\frac{(T_\omega^\ell)''(x)}{((T_\omega^\ell)'(x))^2}
\]
for $x\in (0, 1]$ with $T_\omega^\ell(x) \in (b_{k}^{\sigma^\ell \omega} ,  b_{k-1}^{\sigma^\ell \omega} )$. We have 
\[
\frac{(T_\omega^\ell)''(x)}{((T_\omega^\ell)'(x))^2}=\sum_{j=0}^{\ell-1}\frac{T_{\sigma^j \omega}''(T_\omega^j(x))}{T_{\sigma^j \omega}'(T_\omega^j(x))}\cdot \frac{1}{(T_{\sigma^j \omega}^{\ell-j})'(T_\omega^j(x))} \le \sum_{j=0}^{\ell-1}\frac{T_{\sigma^j \omega}''(T_\omega^j(x))}{(T_{\sigma^j \omega}^{\ell-j})'(T_\omega^j(x))}.
\]
We first consider the case when there exists $0\le j\le \ell$ such that $T_\omega^j (x)\in (\frac 1 2, 1]$.
Let
\[
\left \{0\le j\le \ell: \ T_\omega^j(x)\in (\frac 1 2, 1]\right \}=\{\ell_1, \ldots, \ell_Q\},
\]
for some $Q\in \N$ with $\ell_1<\ldots <\ell_Q$.

\noindent\textbf{Case 1: $0 \le j < \ell_1$}. By the definition of $\ell_1$, we have 
	$T_{\omega}^j(x) \in [ b_{ \ell_1 - j }^{\sigma^j \omega},  b_{ \ell_1 - j - 1 }^{\sigma^j \omega}   ]$. Therefore,
	$$
	T_{\sigma^j \omega}''( T_\omega^j (x)  ) \le C_{\alpha, \underline{\alpha}}  ( b_{ \ell_1 - j }^{\sigma^j \omega})^{ \beta(\sigma^j \omega) - 1 }.
	$$
	Using Lemma~\ref{newlem1} together with~\eqref{first}, we obtain 
	\begin{align}\label{eq:l20-1}
		\begin{split}
		\frac{T_{\sigma^j \omega}''(T_\omega^j(x))}{(T_{\sigma^j \omega}^{\ell-j})'(T_\omega^j(x))} &\le C_{\alpha, \underline{\alpha}} ( b_{\ell_1 - j}^{\sigma^j \omega}   )^{  \beta(\sigma^j \omega) - 1 } 
		\cdot \frac{1}{ (T_{\sigma^j\omega}^{ \ell - j } )'( T_\omega^j(x)  )} \\
		&= C_{\alpha, \underline{\alpha}}( b_{\ell_1 - j}^{\sigma^j \omega}   )^{  \beta(\sigma^j \omega) - 1 }  \frac{1}{  (T_{\sigma^{\ell_1}\omega}^{\ell - \ell_1})' (  T_\omega^{\ell_1}(x)  ) }    \frac{1}{ ( T_{\sigma^j \omega}^{\ell_1 - j} )'( T_\omega^j (x) ) } \\
		&\le C_{p, \alpha, \underline{\alpha}} 
		( b_{\ell_1 - j}^{\sigma^j \omega}   )^{  \beta(\sigma^j \omega) - 1 } \cdot 
		( \ell - \ell_1+1)^{  - 1 / \alpha  } \cdot b_{\ell_1 - j}^{\sigma^j \omega} \\
		&\le C_{p, \alpha, \underline{\alpha}} 
		( b_{\ell_1 - j}^{\sigma^j \omega}   )^{  \beta (\sigma^j \omega) } \cdot 
		( \ell - \ell_1+1)^{  - 1 / \alpha  } \\
		&\le C_{p,\alpha, \underline{\alpha}}  ( \ell_1 - j )^{  - \underline{\alpha} / \alpha } ( \ell - \ell_1+1)^{  - 1 / \alpha  },
		\end{split}
	\end{align}
	except possibly for a finite number of points.
    
\noindent\textbf{Case 2:} $j = \ell_q$ for $1\le q\le Q$. In this case, one has
\[
\frac{T_{\sigma^j \omega}''(T_\omega^j(x))}{(T_{\sigma^j \omega}^{\ell-j})'(T_\omega^j(x))}=0.
\]

\noindent\textbf{Case 3:} $\ell_i < j < \ell_{i + 1}$ for $1 \le i \le Q - 1$. 
	Similarly to Case 1, we have $T^j_\omega(x) \in [ b^{\sigma^j \omega}_{ \ell_{i+1} - j },   b^{\sigma^j \omega }_{ \ell_{i+1} - j - 1} ]$.
Hence, $T_{\sigma^j \omega}''(  T^j_\omega(x) ) \le C_{\alpha, \underline{\alpha}} (  b^{\sigma^j \omega}_{ \ell_{i+1} - j })^{ \beta(\sigma^j \omega) - 1 }$.
	Using Lemma~\ref{newlem1} and~\eqref{first},  we obtain
    \begin{align}\label{eq:l20-2}
	\begin{split}
	 \frac{T_{\sigma^j\omega}''(T_\omega^j(x))}{(T_{\sigma^j \omega}^{\ell-j})'(T_\omega^j(x))}&\le  C_{\alpha, \underline{\alpha}}  (  b^{\sigma^j \omega}_{ \ell_{i+1} - j })^{ \beta(\sigma^j \omega) - 1 } 
	 \frac{1}{ ( T^{ \ell - \ell_{i + 1} }_{\sigma^{ \ell_{i + 1}} \omega})' ( T_\omega^{ \ell_{i+1} } (x) ) } 
	 \frac{1}{ ( T_{\sigma^j \omega}^{ \ell_{i + 1} - j } )' (  T_\omega^j(x)  ) } \\
	 &\le C_{p, \alpha, \underline{\alpha}} \cdot (  b^{\sigma^j\omega}_{ \ell_{i+1} - j })^{ \beta(\sigma^j \omega) - 1 }  \cdot 
	 (1 + \ell -  \ell_{i+ 1} )^{- 1 / \alpha } \cdot b^{\sigma^j \omega}_{ \ell_{i+1} - j } \\
	 &\le C_{p, \alpha, \underline{\alpha}} \cdot (  b^{\sigma^j \omega}_{ \ell_{i+1} - j })^{ \underline{\alpha} }  \cdot 
	 (1 + \ell -  \ell_{i+ 1} )^{- 1 / \alpha } \\
	 &\le C_{p, \alpha, \underline{\alpha}} \cdot (  \ell_{i+ 1} - j )^{ - \underline{\alpha} / \alpha }  \cdot 
	 (1 + \ell -  \ell_{i+ 1} )^{- 1 / \alpha }.
	\end{split}
	\end{align}
	except possibly for a finite number of points.

\noindent\textbf{Case 4:} $\ell_Q < j < \ell$. Since 
	$T_\omega^\ell(x) \in (b_{k}^{\sigma^\ell \omega} ,  b_{k-1}^{\sigma^\ell \omega } ) = (b_k^{\sigma^{j + ( \ell - j) }\omega}, 
	b_{k-1}^{\sigma^{ j + ( \ell - j) }\omega}
	)$ and $\ell_Q$ is the last return 
	by time $\ell$, we must have $T_\omega^{j}(x) \in ( b^{\sigma^j \omega}_{ k + \ell - j }, 
	b^{\sigma^j \omega}_{ k + \ell - j } )$. Hence, again by Lemma~\ref{newlem1} and~\eqref{first}, one has 
	\begin{align}\label{eq:l20-3}
	\begin{split}
	 \frac{T_{\sigma^j \omega}''(T_\omega^j(x))}{(T_{\sigma^j \omega}^{\ell-j})'(T_\omega^j(x))} \le  C_{ \alpha, \underline{\alpha}}   (b^{\sigma^j \omega}_{ k + \ell - j } )^{  \underline{\alpha} - 1  }
	\frac{ 1 }{ (T_{\sigma^j \omega}^{\ell - j})'( T_\omega^j(x)  )}  
	\le  C_{p, \alpha, \underline{\alpha}}  (b^{\sigma^j \omega}_{ k + \ell - j } )^{  \underline{\alpha}  } \le 
	C_{p, \alpha, \underline{\alpha}} ( \ell - j)^{ - \underline{\alpha} / \alpha }.
	\end{split}
\end{align}

We now have that 
\[
\begin{split}
\frac{(T_\omega^\ell)''(x)}{((T_\omega^\ell)'(x))^2} &\le \sum_{j=0}^{\ell-1}\frac{T_{\sigma^j \omega}''(T_\omega^j(x))}{(T_{\sigma^j \omega}^{\ell-j})'(T_\omega^j(x))}\\
&\le C_{p, \alpha, \underline{\alpha}}\sum_{j=0}^{\ell_1-1}(\ell_1-j)^{-\underline{\alpha}/\alpha}+C_{p, \alpha, \underline{\alpha}}\sum_{i=1}^{Q-1}
(1+\ell-\ell_{i+1})^{-1/\alpha}\sum_{\ell_i<j<\ell_{i+1}}(\ell_{i+1}-j)^{-\underline{\alpha}/\alpha}\\
&\phantom{\le}+C_{p, \alpha, \underline{\alpha}}\sum_{\ell_Q<j<\ell}(\ell-j)^{-\underline{\alpha}/\alpha}.\\
\end{split}
\]
Observe that 
\[
\sum_{j=0}^{\ell_1-1}(\ell_1-j)^{-\underline{\alpha}/\alpha}\le C_{\alpha, \underline{\alpha}}\ell^{1-\underline{\alpha}/\alpha}.
\]
Next,
\[
\begin{split}
\sum_{i=1}^{Q-1}
(1+\ell-\ell_{i+1})^{-1/\alpha}\sum_{\ell_i<j<\ell_{i+1}}(\ell_{i+1}-j)^{-\underline{\alpha}/\alpha} &\le C_{\alpha, \underline{\alpha}}\sum_{i=1}^{Q-1} (1+\ell-\ell_{i+1})^{-1/\alpha} (\ell_{i+1}-\ell_i)^{1-\underline{\alpha}/\alpha}  \\
&\le C_{\alpha, \underline{\alpha}}\ell^{1-\underline{\alpha}/\alpha}\sum_{i=1}^{Q-1}(1+\ell-\ell_{i+1})^{-1/\alpha}\\
&\le C_{\alpha, \underline{\alpha}}\ell^{1-\underline{\alpha}/\alpha}\sum_{i=1}^\infty i^{-1/\alpha} \\
&\le C_{\alpha, \underline{\alpha}}\ell^{1-\underline{\alpha}/\alpha}.
\end{split}
\]
Finally, 
\[
\sum_{\ell_Q<j<\ell}(\ell-j)^{-\underline{\alpha}/\alpha}\le C_{\alpha, \underline{\alpha}}\ell^{1-\underline{\alpha}/\alpha}.
\]
The desired conclusion now follows from the last three estimates. 

In the case when $T^j_\omega(x)\le \frac 1 2$ for $0\le j\le \ell$, we have that $T^j(x)\le b_{\ell-j}^{\sigma^j\omega}$ for $0\le j\le \ell$. Therefore,
\[
\begin{split}
\frac{(T_\omega^\ell)''(x)}{((T_\omega^\ell)'(x))^2} \le \sum_{j=0}^{\ell-1}\frac{T_{\sigma^j \omega}''(T_\omega^j(x))}{(T_{\sigma^j \omega}^{\ell-j})'(T_\omega^j(x))}&\le C_{p, \alpha, \underline{\alpha}}\sum_{j=0}^{\ell-1}(b_{\ell-j}^{\sigma^j \omega})^{\beta(\sigma^j \omega)-1}b_{\ell-j+k}^{\sigma^j \omega}\\
&\le C_{p, \alpha, \underline{\alpha}}\sum_{j=0}^{\ell-1} (b_{\ell-j}^{\sigma^j \omega})^{\underline{\alpha}}\\
&\le C_{p, \alpha, \underline{\alpha}}\sum_{j=0}^{\ell-1} (\ell-j)^{-\underline{\alpha}/\alpha}\\
&\le C_{p, \alpha, \underline{\alpha}}\ell^{1-\underline{\alpha}/\alpha},
\end{split}
\]
yielding the desired claim.
    \hfill $\qedsymbol$

\section{Quenched linear response for $L^q$ observables}\label{sec:lqObs}
We begin by establishing an auxiliary result. 
Let $\beta \colon \Omega \to (0, 1)$ be a measurable map such that $\alpha:=\esssup_{\omega \in \Omega}\beta(\omega)<1$. Moreover, let $T_\omega$ be the LSV map with the parameter $\beta(\omega)$ and $\L_\omega$ the transfer operator associated with $T_\omega$. Set $\tilde m$ as the measure on $[0, 1]$ given by $d\tilde m=g\, dm$, where $g(x)=x^{-\alpha}$. Finally, let 
\[
\tilde {\L}_\omega (\varphi)=g^{-1}\L_\omega(g\varphi), \quad \omega \in \Omega, \ \varphi \in L^1(\tilde m).
\]
Observe that 
\[
\tilde{\L}_\omega^n(\varphi)=g^{-1}\L_\omega^n(g\varphi), \quad \omega \in \Omega, \ \varphi \in L^1(\tilde m),
\]
where 
\[
\tilde{\L}_\omega^n:=\tilde {\L}_{\sigma^{n-1}\omega} \circ \ldots \circ \tilde{\L}_\omega.
\]
The proof of the following result is based on the arguments in the proof of~\cite[Proposition 3.5.]{nicol2021large}. We include it for the sake of completeness. 
\begin{proposition}\label{813p}
    Let $\Omega'\subset \Omega$ be a $\sigma$-invariant set such that $\beta(\omega)\le \alpha$ for $\omega \in \Omega'$. Then, the following holds:
    \begin{enumerate}
\item there exists $C_\alpha>0$ such that  for $\varphi \colon [0, 1]\to \R$ bounded, $\omega \in \Omega'$ and $n\in \N$,
\begin{equation}\label{assert1}
    \|\tilde{\L}_\omega^n(\varphi)\|_{L^\infty(\tilde m)}\le C_\alpha \|\varphi\|_{L^\infty(\tilde m)};
\end{equation}
\item for $r\ge 1$ there is $C_{r, \alpha}>0$ such that for $\varphi, \psi \in \C_\ast(a)$  with $m(\varphi)=m(\psi)$, $\omega \in \Omega'$ and $n\in \N$,
\begin{equation}\label{assert2}
    \|\tilde{\L}_\omega^n[g^{-1}(\varphi-\psi)]\|_{L^r(\tilde m)}\le C_{r, \alpha}(\|\varphi\|_{L^1(m)}+\|\psi\|_{L^1(m)})n^{\frac 1 r (-1/\alpha+1)}.
\end{equation}
    \end{enumerate}
\end{proposition}
\begin{proof}
Take $\varphi \colon [0, 1]\to \R$ bounded. Then,
\[
|\tilde{\L}_\omega^n (\varphi)(x)|=x^\alpha |\L_\omega^n (g\varphi)(x)|\le \|\varphi\|_{L^\infty(m)}x^{-\alpha} \L_\omega^n(g)(x)\le \frac{a}{1-\alpha}\|\varphi\|_{L^\infty(m)},
\]
as $\L_\omega^n (g)\in \C_\ast(a)$ (see Remarks~\ref{n06} and~\ref{181}). We conclude that~\eqref{assert1} holds.

Next, we observe that for $r=1$, \eqref{assert2} follows directly from~\eqref{dec} as \[
\|\tilde{\L}_\omega^n[g^{-1}(\varphi-\psi)]\|_{L^1(\tilde m)}=\|\L_\omega^n(\varphi-\psi)\|_{L^1(m)}.
\]
For arbitrary $r\ge 1$, we have 
\begin{equation}\label{x-x}
\|\tilde{\L}_\omega^n[g^{-1}(\varphi-\psi)]\|_{L^r(\tilde m)}\le \|\tilde{\L}_\omega^n[g^{-1}(\varphi-\psi)]\|_{L^\infty(\tilde m)}^{1-1/r} \cdot \|\tilde{\L}_\omega^n[g^{-1}(\varphi-\psi)]\|_{L^1(\tilde m)}^{1/r}.
\end{equation}
Since $\|g^{-1}(\varphi-\psi)\|_{L^\infty(\tilde m)}\le a(m(\varphi)+m(\psi))$, \eqref{assert2} follows from~\eqref{assert1} and the previously discussed case $r=1$.
\end{proof}

Now let us assume that we are in the same setting as in Section~\ref{S:LinRep}.

\begin{theorem}\label{thm:LqLR}
For  $\omega \in \Omega'$ and $\psi\in L^q(\tilde m)$ with $q>\frac{1-\alpha}{1-2\alpha}$, we have that~\eqref{eq:LR} holds.
\end{theorem}

\begin{proof}
The proof is similar to the proof of Theorem~\ref{thm:LR}, so we only give a sketch of the arguments. For $\omega \in \Omega'$, let $L=L(\omega)$ be given by~\eqref{eq:L}. We first show that the series in~\eqref{eq:L} converges. Let $r\ge 1$ be such that $1=\frac 1 r+\frac 1 q$. By Lemma~\ref{estim}, Proposition~\ref{813p} and the  H\"{o}lder inequality,  we have 
\[
\begin{split}
&\sum_{i=1}^\infty \delta(\sigma^{-(i+1)}\omega) \left |  \int_0^1 \psi \, \L_{\sigma^{-i}\omega}^i \left(X_{\beta(\sigma^{-(i+1)} \omega)}N_{\beta(\sigma^{-(i+1)}\omega)} (h(\sigma^{-(i+1)}\omega)) \right)'\, dm\right |\\
&\le \sum_{i=1}^\infty \|\psi \, \L_{\sigma^{-i}\omega}^i \left(X_{\beta(\sigma^{-(i+1)} \omega)}N_{\beta(\sigma^{-(i+1)}\omega)} (h(\sigma^{-(i+1)}\omega)) \right)'\|_{L^1(m)}\\
&\le \sum_{i=1}^\infty \left \|\psi\tilde{\L}_{\sigma^{-i}\omega}^i \left [g^{-1}
 \left(X_{\beta(\sigma^{-(i+1)} \omega)}N_{\beta(\sigma^{-(i+1)}\omega)} (h(\sigma^{-(i+1)}\omega)) \right)' \right ] \right \|_{L^1(\tilde m)}\\
 &\le \|\psi\|_{L^q(\tilde m)}\sum_{i=1}^\infty  \left \|\tilde{\L}_{\sigma^{-i}\omega}^i \left [g^{-1}
 \left(X_{\beta(\sigma^{-(i+1)} \omega)}N_{\beta(\sigma^{-(i+1)}\omega)} (h(\sigma^{-(i+1)}\omega)) \right)' \right ] \right \|_{L^r(\tilde m)} \\
 &\le C_{r, \alpha}\|\psi\|_{L^q(\tilde m)}\sum_{i=1}^\infty i^{\frac 1 r(-1/\alpha+1)}<+\infty,
\end{split}
\]
as $1<\frac{1-\alpha}{r\alpha}$. We conclude that the series in~\eqref{eq:L} converges.

 Arguing as in the proof of Theorem~\ref{thm:LR} and using Proposition~\ref{813p}, we obtain
\[
\int_0^1\psi (h_\epsilon(\omega)-h(\omega))\, dm=\int_0^1  \psi (\L_{\sigma^{-l}\omega}^\epsilon)^l 1\, dm-\int_0^1\psi \L_{\sigma^{-l}\omega}^l1\, dm+O(\|\psi \|_{L^q(\tilde m)}l^{\frac 1 r(1-1/\alpha})),
\]
which yields~\eqref{eq:diffder}. Consequently, again, we need to control the terms~\eqref{E1} and~\eqref{E2}. It follows from Lemma~\ref{HZU}, Proposition~\ref{813p} and the  H\"{o}lder inequality that
\[
\begin{split}
&\sum_{j=1}^{l-1} (\delta (\sigma^{-(j+1)}\omega))^2 \left |\int_0^1 \int_0^1 (1-t) \psi  (\L_{\sigma^{-j}\omega}^\epsilon )^j \partial_\gamma^2 \L_{\beta(\sigma^{-(j+1)}\omega)+t\epsilon \delta(\sigma^{-(j+1)}\omega)}(\L_{\sigma^{-l}\omega}^{l-1-j}(1))\, dm\, dt\right |\\
&\le \sum_{j=1}^{l-1}\int_0^1 \|\psi  (\L_{\sigma^{-j}\omega}^\epsilon )^j \partial_\gamma^2 \L_{\beta(\sigma^{-(j+1)}\omega)+t\epsilon \delta(\sigma^{-(j+1)}\omega)}(\L_{\sigma^{-l}\omega}^{l-1-j}(1))\|_{L^1(m)}\, dt \\
&=\sum_{j=1}^{l-1}\int_0^1 \left \|\psi  (\tilde{\L}_{\sigma^{-j}\omega}^\epsilon )^j \left [g^{-1}\partial_\gamma^2 \L_{\beta(\sigma^{-(j+1)}\omega)+t\epsilon \delta(\sigma^{-(j+1)}\omega)}(\L_{\sigma^{-l}\omega}^{l-1-j}(1))\right ]\right \|_{L^1(\tilde m)}\, dt\\
&\le \|\psi\|_{L^q(\tilde m)}\sum_{j=1}^{l-1}\int_0^1 \left \|  (\tilde{\L}_{\sigma^{-j}\omega}^\epsilon )^j \left [g^{-1}\partial_\gamma^2 \L_{\beta(\sigma^{-(j+1)}\omega)+t\epsilon \delta(\sigma^{-(j+1)}\omega)}(\L_{\sigma^{-l}\omega}^{l-1-j}(1))\right ]\right \|_{L^r(\tilde m)}\, dt \\
&\le C_{r, \alpha}\|\psi\|_{L^q(\tilde m)}\sum_{j=1}^{l-1} j^{\frac 1 r (-1/\alpha+1)},
\end{split}
\]
which implies that~\eqref{E2} converges to $0$ as $\epsilon \to 0$. 

In addition, \eqref{E1} can be estimated as in~\eqref{eq:27dev}. The term~\eqref{eq:27dev}(I) can be estimated as follows:
\[
\begin{split}
&\left |\sum_{j=0}^n \delta(\sigma^{-(j+1)}\omega) \int_0^1\psi ((\L_{\sigma^{-j}\omega}^\epsilon )^j -\L_{\sigma^{-j}\omega}^j)[(X_{\beta(\sigma^{-(j+1)}\omega)} N_{\beta(\sigma^{-(j+1)}\omega)}(\L_{\sigma^{-l}\omega}^{l-1-j}(1)))']\, dm \right| \\
&\le \|\psi\|_{L^q(\tilde m)}  \sum_{j=0}^n \left \|((\tilde{\L}_{\sigma^{-j}\omega}^\epsilon )^j -\tilde{\L}_{\sigma^{-j}\omega}^j)[g^{-1}(X_{\beta(\sigma^{-(j+1)}\omega)} N_{\beta(\sigma^{-(j+1)}\omega)}(\L_{\sigma^{-l}\omega}^{l-1-j}(1)))']\right \|_{L^r(\tilde m)}\\
&\le C_{r, \alpha}\|\psi\|_{L^q(\tilde m)}\sum_{j=0}^n \left \|((\tilde{\L}_{\sigma^{-j}\omega}^\epsilon )^j -\tilde{\L}_{\sigma^{-j}\omega}^j)[g^{-1}(X_{\beta(\sigma^{-(j+1)}\omega)} N_{\beta(\sigma^{-(j+1)}\omega)}(\L_{\sigma^{-l}\omega}^{l-1-j}(1)))']\right \|_{L^1(\tilde m)}^{\frac 1 r}\\
&\le C_{r, \alpha}\|\psi\|_{L^q(\tilde m)}\sum_{j=0}^n \left \|((\L_{\sigma^{-j}\omega}^\epsilon )^j -\L_{\sigma^{-j}\omega}^j)[(X_{\beta(\sigma^{-(j+1)}\omega)} N_{\beta(\sigma^{-(j+1)}\omega)}(\L_{\sigma^{-l}\omega}^{l-1-j}(1)))']\right \|_{L^1( m)}^{\frac 1 r},
\end{split}
\]
where we used Lemma~\ref{estim}, the H\"{o}lder inequality and~\eqref{x-x}. In the proof of Theorem~\ref{thm:LR} it is shown that for each $0\le j \le n$,
\[
\left \|((\L_{\sigma^{-j}\omega}^\epsilon )^j -\L_{\sigma^{-j}\omega}^j)[(X_{\beta(\sigma^{-(j+1)}\omega)} N_{\beta(\sigma^{-(j+1)}\omega)}(\L_{\sigma^{-l}\omega}^{l-1-j}(1)))']\right \|_{L^1( m)} \to 0
\] 
as $\epsilon \to 0$, which implies that 
\[
\left \|((\L_{\sigma^{-j}\omega}^\epsilon )^j -\L_{\sigma^{-j}\omega}^j)[(X_{\beta(\sigma^{-(j+1)}\omega)} N_{\beta(\sigma^{-(j+1)}\omega)}(\L_{\sigma^{-l}\omega}^{l-1-j}(1)))']\right \|_{L^1( m)}^{\frac 1 r} \to 0.
\]
as $\epsilon \to 0$. We conclude that~\eqref{eq:27dev}(I) converges to $0$ as $\epsilon \to 0$. The term~\eqref{eq:27dev}(II) can be treated analogously. 
\end{proof}

\section*{Acknowledgements}
D.D is grateful to J. Lepp{\"a}nen for patiently answering several questions related to~\cite{L} and for many useful comments. D.D is supported in part by University of Rijeka under the project uniri-iskusni-prirod-23-98 3046. 
C.G.T. is supported by the Australian Research Council (ARC), and acknowledges the hospitality and support of the University of Rijeka during a visit in 2023.
The authors acknowledge the referees for their valuable comments and suggestions, which led to significant enhancement of this work.

\end{document}